\documentclass[hidelinks,onefignum,onetabnum]{siamart220329}

\usepackage{lipsum}
\usepackage{amsfonts}
\usepackage{graphicx}
\usepackage{epstopdf}
\usepackage{algorithmicx}
\usepackage{algpseudocode}
\ifpdf
  \DeclareGraphicsExtensions{.eps,.pdf,.png,.jpg}
\else
  \DeclareGraphicsExtensions{.eps}
\fi

\usepackage{amsopn}

\usepackage{amsfonts}
\usepackage{amssymb}
\usepackage{bm}
\usepackage{bbm}
\usepackage{mathtools}
\usepackage{mathrsfs}

\usepackage{tikz}
\usetikzlibrary{calc,positioning}
\tikzstyle{startstop} = [rectangle, rounded corners, minimum width=1.8cm, minimum height=0.7cm,text centered, draw=black, fill=red!30]
\tikzstyle{io} = [trapezium, trapezium left angle=70, trapezium right angle=110, minimum width=3cm, minimum height=1cm, text centered, draw=black, fill=blue!30]
\tikzstyle{process} = [rectangle, minimum width=3cm, minimum height=1cm, text centered, draw=black, fill=orange!30]
\tikzstyle{decision} = [diamond, minimum width=3cm, minimum height=1cm, text centered, draw=black, fill=green!30]
\tikzstyle{arrow} = [thick,->,>=stealth]
\usepackage{tikz-cd}
\usepackage{pgfplots}
\pgfplotsset{compat=1.11}
\usepgfplotslibrary{fillbetween}
\usetikzlibrary{intersections}
\pgfdeclarelayer{bg}
\pgfsetlayers{bg,main}

\newsiamremark{remark}{Remark}
\newsiamremark{hypothesis}{Hypothesis}
\crefname{hypothesis}{Hypothesis}{Hypotheses}
\newsiamthm{claim}{Claim}

\title{A fourth-order cut-cell method for solving 
  the two-dimensional advection-diffusion equation 
  with moving boundaries}

\author{Kaiyi Liang
  \thanks{School of Mathematical Sciences, Zhejiang University,
    38 Zheda Road, Hangzhou, Zhejiang Province, 310027 China 
    (kyliang@zju.edu.cn, yukezhu0323@126.com).}
  \and Yuke Zhu \footnotemark[1]
  \and Jiyu Liu \footnotemark[1]
  \and Qinghai Zhang \footnotemark[1]
  \thanks{Corresponding author. 
    Institute of Fundamental and Transdisciplinary Research,
    Zhejiang University, Hangzhou, Zhejiang, 310058, China 
    (\email{qinghai@zju.edu.cn}).}}

\newcommand{\bi}{{\mathbf{i}}}
\newcommand{\bj}{{\mathbf{j}}}
\newcommand{\bk}{{\mathbf{k}}}
\newcommand{\be}{{\mathbf{e}}}
\newcommand{\bu}{{\mathbf{u}}}
\newcommand{\bv}{{\mathbf{v}}}
\newcommand{\bx}{{\mathbf{x}}}

\newcommand{\bp}{{\mathbf{p}}}
\newcommand{\bn}{{\mathbf{n}}}
\newcommand{\adv}{\mathbf{L}_\mathrm{adv}}
\newcommand{\lap}{\mathbf{L}_\mathrm{lap}}
\newcommand{\eadv}{\mathcal{L}_\mathrm{adv}}
\newcommand{\elap}{\mathcal{L}_\mathrm{lap}}
\newcommand{\dif}{\mathrm{d}}
\newcommand{\vol}[1]{\left \| #1 \right \|}
\newcommand{\avg}[1]{\left\langle #1 \right\rangle}
\makeatletter
\newcommand{\fdsy@scale}{1.0}
\newcommand\fdsy@mweight@normal{Book}%
\newcommand\fdsy@mweight@small{Regular}%
\newcommand\fdsy@bweight@normal{Medium}%
\newcommand\fdsy@bweight@small{Bold}%
\DeclareFontFamily{U}{FdSymbolF}{}
\DeclareFontShape{U}{FdSymbolF}{m}{n}{
    <-7.1> s * [\fdsy@scale] FdSymbolF-\fdsy@mweight@small
    <7.1-> s * [\fdsy@scale] FdSymbolF-\fdsy@mweight@normal
}{}
\DeclareFontShape{U}{FdSymbolF}{b}{n}{
    <-7.1> s * [\fdsy@scale] FdSymbolF-\fdsy@bweight@small
    <7.1-> s * [\fdsy@scale] FdSymbolF-\fdsy@bweight@normal
}{}
\makeatother
\DeclareSymbolFont{delimiters}{U}{FdSymbolF}{m}{n}
\SetSymbolFont{delimiters}{bold}{U}{FdSymbolF}{b}{n}
\DeclareMathDelimiter{\lAngle}{\mathopen}{delimiters}{"92}{delimiters}{"92}
\DeclareMathDelimiter{\rAngle}{\mathclose}{delimiters}{"98}{delimiters}{"92}
\newcommand{\bbrk}[1]{\left\lAngle #1 \right\rAngle}
\newcommand{\bbR}{\mathbb{R}}
\newcommand{\YIN}{\mathbb{Y}}
\newcommand{\MARS}{{\mathfrak L}_{\textup{MARS}}}
\newcommand{\app}{\mathrm{app}}
\newcommand{\Tiny}{\mathrm{tiny}}
\newcommand{\bC}{\mathbf{C}}
\newcommand{\bF}{\mathbf{F}}
\newcommand{\bN}{\mathbf{N}}
\newcommand{\hfIdx}{\bi+\frac{1}{2}\be^d}

\usepackage{subcaption}
\usepackage{enumerate}
\usepackage{xcolor}

\usepackage{bookmark}

\ifpdf
\hypersetup{
  pdftitle={A fourth-order cut-cell method for solving 
  the two-dimensional advection-diffusion equation with moving boundaries},
  pdfauthor={K. Liang, Y. Zhu, J. Liu, Q. Zhang}
}
\fi

\begin{document}
\begin{sloppypar}
\maketitle

\begin{abstract}
We propose a fourth-order cut-cell method 
 for solving the two-dimensional advection-diffusion equation 
 with moving boundaries 
 on a Cartesian grid. 
We employ the ARMS technique 
 to give an explicit and accurate representation of moving boundaries, 
 and introduce a cell-merging technique to overcome 
 discontinuities caused by topological changes in cut cells 
 and the small cell problem. 
We use a polynomial interpolation technique 
 base on poised lattice generation  
 to achieve fourth-order spatial discretization, 
 and use a fourth-order implicit-explicit Runge-Kutta scheme 
 for time integration. 
Numerical tests are performed on various moving regions, 
 with advection velocity both matching and differing from boundary velocity, 
 which demonstrate the fourth-order accuracy of the proposed method. 
 
\end{abstract}



\section{Introduction}
\label{sec:Introduction}
The advection-diffusion equation 
 is a fundamental partial differential equation (PDE) 
 used to model a variety of physical phenomena, 
 such as air pollution~\cite{Sharan2003} and heat transfer~\cite{Fei2018}. 
In this study, 
 we focus on the two-dimensional advection-diffusion equation 
 with moving boundaries, 
 the dimensionless form is given by 
\begin{equation}
    \label{eq:AdvDiff}
    \begin{aligned}
    \dfrac{\partial \rho(\bx, t)}{\partial t} 
        + \bu(\bx, t) \cdot \nabla\rho(\bx, t) &= 
        \dfrac{1}{\mathrm{Pe}} \Delta \rho(\bx, t) 
        && \mathrm{in}\, \Omega(t)\times[0,T],\\
    \rho(\bx, t_0) &= \rho_\mathrm{init}(\bx) 
        && \mathrm{in}\, \Omega(0),\\
    \rho(\bx,t) &= \rho_\mathrm{bc}(\bx, t) 
        && \mathrm{on}\, \partial\Omega(t)\times[0,T],
    \end{aligned}
\end{equation}
where $\Omega(t)$ is a time-dependent domain driven by a velocity field 
 $\bv(\bx, t) \coloneqq (v_1(\bx, t), v_2(\bx, t))$, 
 $\partial \Omega(t)$ its boundary, 
 $\bx$ the spatial position, $t\in[0,T]$ the time, 
 $\bu(\bx, t) \coloneqq (u_1(\bx, t), u_2(\bx, t))$ the advection velocity, 
 $\rho(\bx, t)$ the unknown variable, 
 $\rho_\mathrm{init}(\bx)$ the initial condition, 
 $\rho_\mathrm{bc}(\bx, t)$ the boundary conditions, 
 and $\mathrm{Pe}$ the Peclet number. 

In recent decades, 
 significant progress has been made in developing numerical methods 
 for solving PDEs on complex domain with moving boundaries. 
These methods can be classified into two main catagories: 
 conforming and non-conforming mesh methods. 
In conforming mesh methods, 
 the mesh adapts to the domain boundaries as they evolve in time, 
 allowing straightforward application of boundary conditions 
 and high-order accuracy. 
However, owing to the large translation and deformation of the boundaries, 
 special treatments such as re-meshing are needed, which is time consuming. 
Popular conforming mesh methods include 
 Arbitrary Lagrangian-Eulerian (ALE) methods~\cite{
    Boiarkine2011ConvDiffALE, Formaggia2004ALE} 
 and particle finite element methods (PFEM)~\cite{
    Cremonesi2020PFEM_Review, Idelsohn2006PFEM}. 

Non-conforming mesh methods, on the other hand, use fixed Cartesian grids, 
 providing simpler grid generation and improved computational efficiency. 
Immersed boundary method (IBM)~\cite{
    Kim2006IBM,Liao2010IBM_MvBdry,Peskin2002IBM}
 was first introduced by Peskin~\cite{Peskin1972IBM} 
 to study flow patterns around heart valves. 
In IBM, a Cartesian grid is used for the entire fluid domain, 
 and the boundary of the immersed elastic object 
 is treated as a set of Lagrangian fluid particle. 
The effect of the object is represented by a force exerted on the fluid, 
 which is added to the model equation 
 by approximating it through a discrete Dirac delta function. 
This approach eliminates the need for costly mesh updates, 
 but is limited to structures without volumes.  
Fictitious domain method~\cite{Glowinski2001, Yu2006} 
 and Immersed finite element method (IFEM)~\cite{
    ZhangLT2007IFEM, ZhangLucy2004IFEM} 
 address the limitation of IBM 
 by treating the union of fluid phase and solid phase as one continuum 
 and using a discrete Dirac delta function 
 to couple the velocity fields in fluid phase and solid phase. 
In above methods, 
 a Dirac delta function is needed to handle the discontinuity 
 between the fluid and the solid phases, 
 however, the discrete approximation of the Dirac delta function 
 may limit the accuracy.
Immersed interface methods (IIM)~\cite{
    LiZhilin1997IMM_MvInterface, Russel2003IIM, Xu2006IIM_MvBdry}, 
 use an alternative way to deal with this discontinuity. 
IIM was first developed by LeVeque and Li~\cite{Leveque1994IIM, LeVeque1997IIM} 
 to solve elliptic interface problems, 
 the key idea is to incorporate the jump conditions 
 caused by the Dirac delta function into finite difference schemes, 
 thus avoiding the approximation and achieving high-order accuracy. 

In this work, we focus on the cut-cell finite volume method, 
 which calculates the geometries of boundary cells directly 
 within the non-conforming framework. 
This method allows for higher accuracy near the boundary, 
 but requires an interface tracking (IT) method 
 to represent the moving boundary accurately. 
Most methods~\cite{Barrett2022, Meinke2013, Schneiders2013} 
 employ the level-set method to track the interface, 
 which defines a signed-distance function $\zeta(\bx, t)$ and 
 use the zero level-set $\{\bx : \zeta(\bx, t) = 0\}$ 
 to represent the boundary. 
While the computation is convenient,  
 level-set method requires reconstructing 
 an explicit boundary representation at each time step, 
 which is hard to achieve high-order accuracy. 
Another challenge in processing moving boundary problems 
 arises from topological changes. 
For example, an interface cell may become a pure solid cell 
 or a pure fluid cell in a time step (or vice-versa). 
These transitions introduce discontinuities in cell-averaged values, 
 limiting the cut-cell methods' accuracy. 
Schneiders et al.~\cite{Schneiders2013} proposed a method 
 using weight functions to make the discrete operator change smoothly 
 during these transitions, 
 but it still cannot ensure high-order continuity. 
To our acknowledge, 
 no current cut-cell methods have resolve this issue 
 to achieve accuracy beyond second-order 
 when handling moving boundary problems. 
Additionally, 
 cut-cell methods suffer from the so-called ``small cell problem'', 
 where extremely small time steps are required for stability 
 in cells with small volumes 
 due to the Courant–Friedrichs–Lewy (CFL) condition. 
Approaches to address this issue include 
 cell merging technique~\cite{Causon2001Cell_Merging},
 the h-box method~\cite{Helzel2005H_Box}
 and flux-redistribution technique~\cite{Schneiders2013}. 

These difficulties give rise to the following questions:
\begin{enumerate}[({Q}-1)]
    \item \label{Q-1} Can we represent the moving boundary accurately 
     in the framework of cut-cell methods;
    \item \label{Q-2} Can we get rid off the accuracy limitation 
     caused by topological changes;
    \item \label{Q-3} Can we achieve high-order accuracy 
     when handling complex moving regions.
\end{enumerate}

In this study, we give positive answer to all these questions. 
For (Q-\ref{Q-1}), 
 we employ the ARMS technique~\cite{Hu2024ARMS} 
 to track the moving boundary. 
This method achieves a fourth-order explicit representation 
 of the boundary at any given time, 
 and the moving region is modeled by a Yin set~\cite{Zhang2020YinSet} 
 equipped with efficient boolean operators, 
 which is employed in this work to cut cells. 
For (Q-\ref{Q-2}), 
 we analyze the accuracy limitations 
 in solving moving boundary problem with cut-cell methods, 
 and introduce a cell-merging technique, 
 which processes all cases that will introduce discontinuities 
 as long as the presence of small cells. 
For (Q-\ref{Q-3}), 
 we derive standard discretization formula for cells far from boundaries 
 and use polynomial interpolation technique 
 based on poised lattice generation (PLG)~\cite{Zhang2024PLG} 
 for cells near boundaries 
 to give a fourth-order accurate spatial discretization. 
We also use a fourth-order implicit-explicit Runge-Kutta (IMEX) scheme 
 for time integration. 
Since the accuracy limitations are no longer exist, 
 we can achieve fourth-order accuracy in both space and time. 

The rest of this paper is organized as follows. 
A review of key prerequisites for our algorithm is provided 
 in \Cref{sec:Preliminaries}. 
In \Cref{sec:Algorithm}, 
 we introduce our algorithm for solving 
 the two-dimensional advection-diffusion equation 
 with moving boundaries. 
Various numerical tests are conducted in \Cref{sec:Numerical_results} 
 to demonstrate the fourth-order accuracy of our method. 
Finally, we conclude this work in \Cref{sec:Conclusions}.

\section{Preliminaries}
\label{sec:Preliminaries}
In this section, 
 we briefly review the Yin sets~\cite{Zhang2020YinSet} 
 as a theoretical model of continua, 
 the MARS framework~\cite{Zhang2016MARS} 
 which unifies explicit IT methods and provides an error analysis technique, 
 and the ARMS strategy~\cite{Hu2024ARMS} 
 as a high-order IT method for solving general IT problems. 
These methods are employed in this work to model and track the moving region 
 and to generate cut cells at any given time. 
Additionally, we review the Reynolds transport theorem 
 to facilitate the finite volume discretization 
 and the analysis of the discontinuities caused by topological changes. 

\subsection{Yin sets that model continua}
In a topological space ${\cal X}$,
 denote by ${\cal P}'$, $\overline{{\cal P}}$, and ${\cal P}^{\circ}$ 
 the complement, the closure, and the interior 
 of a subset ${\cal P}\subseteq{\cal X}$, respectively.
The exterior of ${\cal P}$, 
 written ${\cal P}^{\perp}\coloneqq({\cal P}')^{\circ}$,
 is the interior of its complement. 
An open set ${\cal P}\subseteq{\cal X}$ is \emph{regular} 
 if ${\cal P}=\overline{\cal P}^{\circ}$.
A closed set ${\cal P}\subseteq{\cal X}$ is \emph{regular}
 if ${\cal P}=\overline{\mathcal P^{\circ}}$.

Regular sets, open or closed, capture the salient feature 
 that regions occupied by continua are free of lower-dimensional elements
 such as isolated points and curves in $\mathbb{R}^2$.  
However, as shown in \cite[eqn (3.1)]{Zhang2020YinSet}, 
 the intersection of two regular sets
 might contain an infinite number of connected components,
 making it difficult to store regular sets on computer.
Therefore, we need to further add the restriction of semianalytic sets
 to reach a subspace of regular sets 
 where each element can be described by a finite number of entities.

\begin{definition}
    \label{def:semianalytic}
    A set ${\cal P}\subseteq \bbR^{2}$ is \emph{semianalytic}
     if there exist a finite number of analytic functions $g_i:\bbR^{2}\to\bbR$
     such that ${\cal P}$ is in the universe of a finite Boolean algebra 
     formed from the sets $\left\{\bx\in\bbR^{2}:g_i(\bx)\geq 0\right\}$.
    The $g_i$'s are called the generating functions of ${\cal P}$.
\end{definition} 

As illustrated in \cite[Figure 5]{Zhang2020YinSet},
 a regular \emph{open} set can be uniquely represented 
 by its boundary Jordan curves
 while a regular \emph{closed} set can not.
This leads to

\begin{definition}[Yin space~\cite{Zhang2020YinSet}]
  \label{def:Yin_space}
  A \emph{Yin set} ${\cal Y}\subseteq\bbR^{2}$ is a regular open semianalytic set 
  whose boundary is bounded. 
  All Yin sets form the \emph{Yin space} $\YIN$.
\end{definition}

\begin{theorem}[Global topology of Yin sets~\cite{Zhang2020YinSet}]
    \label{thm:Unique_representation}
    The boundary of a connected Yin set 
     \mbox{${\cal Y}\ne \emptyset, \mathbb{R}^2$}
     can be uniquely partitioned into 
     a finite set of pairwise almost disjoint Jordan curves, 
     which can be uniquely oriented so that
    \begin{equation*}
        {\cal Y} = \bigcap_{\gamma_{j}\in{\cal J}_{\partial {\cal Y}}} 
                 \mathrm{int}(\gamma_j),
    \end{equation*}
     where the \emph{interior of an oriented Jordan curve},
     written $\mathrm{int}(\gamma)$,
     is the component of the complement of $\gamma$ 
     that always lies to the left of the observer 
     who traverses $\gamma$ according to its orientation.
    ${\cal J}_{\partial{\cal Y}}$ is 
     the set of oriented boundary Jordan curves of ${\cal Y}$.
\end{theorem}

Theorem \ref{thm:Unique_representation} furnishes 
 the unique and efficient representation of any Yin set
 as a partially ordered set of oriented Jordan curves~\cite{Zhang2020YinSet}.
The pairwise almost disjointness of Jordan curves
 precludes crossing intersections of any two curves, 
 and this obviates the need of treating self-intersections 
 for practical applications.
 
Following \cite{Hu2024ARMS},
 the unique representation of any Yin set ${\cal Y}$ 
 as a partially ordered set of oriented Jordan curves
 is realized in two steps.
First, each Jordan curve $\gamma_j\subset {\cal J}_{\partial{\cal Y}}$ 
 is reduced to a finite sequence $(\mathbf{X}^j_i)_{i=0}^{N}\subset\gamma_j$
 of characteristic points or markers,
 whose numbering is consistent with the orientation of $\gamma_j$. 
Second, a piecewise smooth curve
 is generated from $(\mathbf{X}^j_i)_{i=0}^{N}\subset \gamma_j$
 to approximate $\gamma_j$. 
In this work, 
 the curve fitting scheme is periodic cubic spline fitting. 

Yin sets also form a Boolean algebra
 with respect to regularized operations.

\begin{theorem}
    \label{thm:BooleanAlgebra}
    The universal algebra 
     $\mathbf{Y}\coloneqq(\YIN,\ \cup^{\perp\perp},\ \cap,\ \, ^{\perp},\
        \emptyset,\ \bbR^2)$
     is a Boolean algebra,
     where the binary operation of regularized union is given by 
     ${\cal P}\cup^{\perp\perp}{\cal Q}\coloneqq
        ({\cal P}\cup{\cal Q})^{\perp\perp}$
     for all ${\cal P},\,{\cal Q}\in\YIN$.
\end{theorem}

Efficient algorithms of Boolean operations on Yin sets 
 have been developed in~\cite{Zhang2020YinSet} 
 and are employed in this work to cut cells.

\subsection{The MARS framework}
Yin sets only model stationary continua. 
For movement of a continuum, 
 we first introduce the flow map. 
Under a given velocity field $\bv$, 
 the continuum is evolved by the ordinary differential equation (ODE)
\begin{equation}
   \label{eq:ODE_IT}
   \frac{\dif \bp}{\dif t} = \bv(\bp,t),
\end{equation}
 where $t$ is the time and $\bp(t)$ is the position of a Lagrangian particle 
 inside the continuum.
Given that $\bv(\bx,t)$ is continuous in time 
 and Lipschitz continuous in space, 
 \cref{eq:ODE_IT} admits an unique solution 
 for any given initial time and initial position. 
This uniqueness furnishes an (exact) \emph{flow map} 
 $\phi:\bbR^{2}\times\bbR\times\bbR\to\bbR^{2}$,
\begin{equation}
   \label{eq:Flow_map}
   \phi_{t_0}^{\pm\tau}(\bp) \coloneqq \bp(t_0\pm\tau) = 
   \bp(t_0) + \int_{t_0}^{t_0\pm\tau} \bv(\bp(t),t)\ \dif t.
\end{equation}
 where the three independent variables 
 $\bp(t_0)$, $t_0$, and $\tau\ge 0$ 
 are the initial position, the initial time, and the time increment, 
 respectively. 
The flow map leads to 

\begin{definition}[General IT]
    \label{def:Interface_tracking}
    The \emph{problem of general IT} 
     is to determine a Yin set ${\cal M}(t)\in \YIN$, c.f. \cref{def:Yin_space},
     from the initial condition ${\cal M}(t_0)\in \YIN$ 
     and a continuous flow map 
     $\phi_{t_0}^{\tau}: \YIN\times [0,t-t_0]\to\YIN$.
\end{definition}

Zhang and Fogelson~\cite{Zhang2016MARS} have proposed a framework 
 for solving the problem of general IT. 

\begin{definition}[MARS method~\cite{Zhang2016MARS}]
    \label{def:MARS}
    A \emph{MARS method} is a general IT method of the form
    \begin{equation}
      \label{eq:MARS}
      {\cal M}^{n+1}=\MARS^n {\cal M}^n \coloneqq 
      \left(\chi_{n+1}\circ\varphi_{t_n}^k\circ\psi_n\right){\cal M}^n,
    \end{equation}
     where ${\cal M}^n\in \YIN$ is the approximation 
     of ${\cal M}(t_n)\in \YIN$, 
     $\varphi_{t_n}^k: \YIN\to\YIN$ is a fully discrete mapping operation 
     that approximates the exact flow map $\phi$ both in time and in space,
     $\psi_n:\YIN\to\YIN$ is an augmentation operation at $t_n$ 
     to prepare ${\cal M}^n$ for $\varphi_{t_n}^k$, 
     and $\chi_{n+1}:\YIN\to\YIN$ is an adjustment operation 
     after the mapping $\varphi_{t_n}^{k}$.
\end{definition}

The three operations in \cref{eq:MARS} 
 provide much freedom on the design of IT algorithms
and consequently 
Definition \ref{def:MARS} can be used to unify explicit IT methods including 
 the VOF method~\cite{Hirt1981VOF},
 the MOF method~\cite{Ahn2007MOF},
 the front tracking method~\cite{Tryggvason2001FT}, 
 and the iPAM method~\cite{Zhang2014IPAM}.
We also developed fourth- and higher-order 
 cubic MARS methods~\cite{Zhang2018CubicMARS}
 under the guidance of the MARS framework \cite{Zhang2016MARS}.

Denote by $\vol{\cal Y}\coloneqq\left|\int_{\cal Y}\dif\bx\right|$ 
 the norm for all ${\cal Y}\in\YIN$ 
 and $E_1(t_n)\coloneqq\vol{{\cal M}(t_n)\oplus{\cal M}^{n}}$ 
 the IT error at $t_n$, 
 where ${\cal M}(t_n)$ and ${\cal M}^{n}$ are respectively 
 the exact solution and the computational result at $t_n$, 
 and ``$\oplus$'' denotes the \emph{symmetric difference} 
 or \textit{exclusive disjunction} given by 
 ${\cal P}\oplus{\cal Q}\coloneqq
    ({\cal P}\setminus{\cal Q})\cup^{\perp\perp}({\cal Q}\setminus{\cal P})$. 
The overall error $E_1$ can be divided into a number of individual error terms 
 that correspond to the three operations of the MARS method in \cref{eq:MARS}.

\begin{definition}
    \label{def:MARS_errors}
    Individual errors of a MARS method include the following:
    \begin{itemize}
    \item The \emph{representation error} $E^{\textup{REP}}(t_n)$ 
     is the final IT error caused by 
     approximating the region ${\cal M}(t_0)$ with a Yin set ${\cal M}^0$
     at the initial time.
    \item The \emph{augmentation error} $E^{\textup{AUG}}(t_n)$
     is the accumulated error of augmenting the Yin sets by $\psi_j$.
    \item The \emph{mapping error} $E^{\textup{MAP}}(t_n)$ 
     is the accumulated error of approximating the exact flow map $\phi$
     with the fully discrete flow map $\varphi$.
    \item The \emph{adjustment error} $E^{\textup{ADJ}}(t_n)$
     is the accumulated error of adjusting the mapped Yin sets by $\chi_{j+1}$.
    \end{itemize}
\end{definition}

At $t_n=t_0+nk$, these individual error terms are given by
\begin{subequations}
  \label{eq:MARS_errors}
  \begin{align}
    E^{\textup{REP}}(t_n)&\coloneqq
     \vol{\phi^{nk}_{t_0}\left[{\cal M}(t_0),{\cal M}^0\right]};\\
    E^{\textup{AUG}}(t_n)&\coloneqq
     \vol{\bigoplus_{j=0}^{n-1}\phi_{t_j}^{(n-j)k}
     \left[{\cal M}^j_\psi,{\cal M}^j\right]};\\
    E^{\textup{MAP}}(t_n)&\coloneqq
     \vol{\bigoplus_{j=1}^{n}\phi_{t_j}^{(n-j)k}
     \left[\phi_{t_{j-1}}^k{\cal M}^{j-1}_\psi,
     \varphi^k_{t_{j-1}}{\cal M}^{j-1}_\psi\right]};\\
    E^{\textup{ADJ}}(t_n)&\coloneqq
     \vol{\bigoplus_{j=1}^{n}\phi_{t_j}^{(n-j)k}
     \left[\varphi_{t_{j-1}}^k{\cal M}^{j-1}_\psi, {\cal M}^{j}\right]},
  \end{align}
\end{subequations}
 where ${\cal M}^j_\psi\coloneqq\psi_j{\cal M}^j$
 and $\phi_{t_j}^{(n-j)k}[\cdot,\cdot]$ is a shorthand
 given by
\begin{equation*}
  \forall {\cal P},\ {\cal Q}\in \mathbb{Y},\quad
  \phi_{t_0}^\tau[{\cal P}, {\cal Q}]\coloneqq
   \phi_{t_0}^\tau({\cal P})\oplus\phi_{t_0}^\tau ({\cal Q}).
\end{equation*}

It is emphasized that the individual errors in \cref{def:MARS_errors} 
 are from~\cite{Hu2024ARMS}, 
 which is a modified version 
 of the original definition~\cite[Definition 4.4]{Zhang2016MARS}, 
 see \cite[Section 3.2]{Hu2024ARMS} for details. 

\begin{theorem}
    \label{thm:errorBound}
    The error $E_1(t_n)$ of a MARS method for solving the general IT problem
    in \cref{def:Interface_tracking}
    is bounded as 
    \begin{equation}
      \label{eq:bound}
      E_1(t_n)\leq 
      E^{\textup{REP}}(t_n)+E^{\textup{AUG}}(t_n)
      +E^{\textup{MAP}}(t_n)+E^{\textup{ADJ}}(t_n),
    \end{equation}
    where the individual errors are defined in \cref{eq:MARS_errors}.
\end{theorem}
\begin{proof}
    See \cite[Theorem 3.3]{Hu2024ARMS}.
\end{proof}

\subsection{The ARMS strategy}
The following ARMS strategy consists of 
 periodic spline fitting as the curve fitting scheme 
 and a specific regularization scheme of markers, 
 so that a definition of the discrete flow map $\varphi_{\app}$ 
 immediately yields a MARS method.

\begin{definition}[The ARMS strategy~\cite{Hu2024ARMS}]
    \label{def:ARMS}
    Suppose we are given a fully discrete flow map $\varphi_{\app}$
     that may or may not depend on a curve fitting scheme. 
    Denote by $[r_{\Tiny}h_L, h_L]$ a range of allowed distances 
     between adjacent interface markers with $r_{\Tiny}<\frac{1}{3}$.
    Within each time step $[t_n,t_n+k]$,
     \emph{the ARMS strategy} takes as its input a set of periodic splines 
     representing $\partial{\cal M}(t_n)$
     and advances it to that representing $\partial{\cal M}^{n+1}$:
    \begin{enumerate}[({ARMS}-1)]
    \item Trace forward in time breakpoints of $\partial{\cal M}^n$ at $t_n$ 
     to a sequence of points $(\bp_j)_{j=0}^{N}$ at $t_{n}+k$ 
     with $\bp_0=\bp_{N}$ by $\varphi_{\app}$. 
    \item If any chordal length $\|\bp_j-\bp_{j+1}\|_2$ 
     is greater than $h_L^*\coloneqq(1-r_{\Tiny})h_L$, 
     then
    \begin{enumerate}[(a)]
        \item locate $\overleftarrow{\bp_j}=(x(l_j), y(l_j))$ 
         and $\overleftarrow{\bp_{j+1}}=(x(l_{j+1}), y(l_{j+1}))$
         on $\partial{\cal M}^n(l)$
         as the preimages of $\bp_j$ and $\bp_{j+1}$,
        \item divide the parameter interval $[l_j, l_{j+1}]$ 
         into $\left\lceil\frac{\|\bp_j-\bp_{j+1}\|_2}{h_L^*}\right\rceil$ 
         equidistant subintervals, 
         compute the corresponding new markers on $\partial{\cal M}^n(l)$, 
         insert them in between 
         $\overleftarrow{\bp_j}$ and $\overleftarrow{\bp_{j+1}}$, and
        \item trace forward in time the new sequence of interface markers 
         on $\partial{\cal M}^n$ to their images at $t_n+k$
         by the flow map $\varphi_{\app}$.
    \end{enumerate}
    Repeat the above substeps until no chordal length is greater than $h_L^*$.
    \item Traverse the point sequence $(\bp_j)_{j=0}^N$ 
     to remove chords of negligible lengths: 
    \begin{enumerate}[(a)]
        \item locate a point $\bp_s$ satisfying
         $\|\bp_s-\bp_{s+1}\|_2\ge r_{\Tiny}h_L$ and set $j=s$; 
         set $j=0$ if such a point does not exist, 
        \item if $\|\bp_j-\bp_{j+1}\|_2<r_{\Tiny}h_L$, 
         keep removing $\bp_{j+1}$ from the point sequence 
         until $\|\bp_j-\bp_{j+1}\|_2\ge r_{\Tiny}h_L$ holds 
         for the new $\bp_{j+1}$, 
      \item increment $j$ by 1 and repeat {(ARMS-3b)}, 
      \item terminate after all pairs of adjacent markers have been checked.
    \end{enumerate}
    \item Construct a new set of splines $\partial{\cal M}^{n+1}$
      from the updated marker sequences 
      as the representation of $\partial{\cal M}(t_n+k)$.
    \end{enumerate}
\end{definition}

As shown in~\cite[Section 5]{Hu2024ARMS}, 
 for solving the mean curvature flow of a Jordan curve, 
 if the discrete flow map $\varphi_{\app}$ is derived using 
 an $r$-th order spatial discretization 
 and a $q$-th order ODE solver for time integration, 
 combined with an $r$-th order accurate periodic spline for curve fitting, 
 the overall IT error of the ARMS method is $E_1(t_n)=O(h_L^r) + O(k^q)$, 
 where $r=2,4,6$. 

In this work, the velocity field $\bv$ is given a priori, 
 which simplifies the IT problem compared to solving the mean curvature flow. 
By employing the classical fourth-order Runge-Kutta method 
 to solve the ODE \cref{eq:ODE_IT} as the discrete flow map $\varphi_{\app}$ 
 and periodic cubic splines for curve fitting scheme, 
 the overall IT error of the ARMS method becomes $E_1(t_n)=O(h_L^4) + O(k^4)$. 
The proof follows the same approach 
 as described in~\cite[Section 5]{Hu2024ARMS}. 

\subsection{Reynolds transport theorem}
We introduce the Reynolds transport theorem 
 to prepare the finite volume discretization 
 and analysis of the discontinuities cause by topological changes. 
\begin{theorem}[Reynolds transport theorem]
    \label{thm:RTT}
    Denote by $V(t)$ a connected moving fluid region 
     driven by a velocity field $\bv$, 
     $f(\bx, t)$ a $C^1$-continuous scalar function, 
     we have 
    \begin{equation}
        \label{eq:RTT}
        \dfrac{\dif}{\dif t}\iint_{V(t)} f(\bx, t)\,\dif \bx = 
         \iint_{V(t)}\dfrac{\partial f}{\partial t}(\bx, t)\,\dif \bx 
         + \oint_{\partial V(t)} f(\bx,t)\bv\cdot\bn\,\dif l,
    \end{equation}
     where $l$ is the arc-length parameter, 
     and $\bn$ is the unit outward normal vector.
\end{theorem}

\section{Algorithm}
\label{sec:Algorithm}
We embed the moving region $\Omega(t)$ inside a rectangular region 
 $\Omega_{R} \in \YIN_{c}$ 
 and divide $\Omega_{R}$ by structured rectangular grids 
 into control volumes or cells, 
\begin{equation*}
    \bC_{\bi} \coloneqq (\bi h, (\bi + \mathbbm{1})h),
\end{equation*}
 where $h$ is the uniform grid size, 
 $\bi \in \mathbb{Z}^2$ a multi-index, 
 and $\mathbbm{1} \in \mathbb{Z}^2$ the multi-index 
 with all components equal to one. 
The higher face of $\bC_{\bi}$ in dimension $d$ is 
\begin{equation*}
    \bF_{\hfIdx} \coloneqq ((\bi+\be^d)h, (\bi+\mathbbm{1})h),
\end{equation*}
 where $\be^d \in \mathbb{Z}^2$ is the multi-index 
 with $1$ in its $d$th component and $0$ in the other component. 
The lower-left corner node of $\bC_{\bi}$ is 
\begin{equation*}
    \bN_{\bi} \coloneqq \bi h.
\end{equation*}

For the moving region $\Omega(t)$, 
 the $\bi$th \emph{cut cell} at time $t$ is defined by 
 ${\cal C}_{\bi}(t) \coloneqq \bC_{\bi} \cap \Omega(t)$, 
 the higher \emph{cut face} of ${\cal C}_{\bi}(t)$ by 
 ${\cal F}_{\hfIdx}(t) \coloneqq \bF_{\hfIdx} \cap \Omega(t)$, 
 the $\bi$th \emph{cut boundary} by 
 ${\cal S}_{\bi}(t) \coloneqq \bC_{\bi} \cap \partial \Omega(t)$. 

A cut cell is classified as an \emph{empty cell}, 
 a \emph{pure cell}, or an \emph{interface cell} 
 if it satisfies ${\cal C}_{\bi} = \emptyset$, ${\cal C}_{\bi} = \bC_{\bi}$, 
 or otherwise, respectively. 

The averaged value of $\rho(\bx, t)$ over a nonempty cut cell ${\cal C}_{\bi}(t)$, 
 a nonempty cut face ${\cal F}_{\hfIdx}(t)$ and a cut boundary ${\cal S}_{\bi}(t)$ 
 are 
\begin{align*}
    \avg{\rho}_{\bi}(t) &\coloneqq 
     \frac{1}{\vol{{\cal C}_{\bi}(t)}}
     \iint_{{\cal C}_{\bi}(t)} \rho(\bx, t)\ \dif \bx; \\
    \avg{\rho}_{\hfIdx}(t) &\coloneqq 
     \frac{1}{\vol{{\cal F}_{\hfIdx}(t)}}
     \int_{{\cal F}_{\hfIdx}(t)} \rho(\bx, t)\ \dif l; \\
    \bbrk{\rho}_{\bi}(t) &\coloneqq 
     \frac{1}{\vol{{\cal S}_{\bi}(t)}} 
     \int_{{\cal S}_{\bi}(t)} \rho(\bx, t)\ \dif l,
\end{align*}
 respectively, 
 where $\vol{{\cal C}_{\bi}(t)}$ is the volume of ${\cal C}_{\bi}(t)$, 
 and $\vol{{\cal F}_{\hfIdx}(t)}$, $\vol{{\cal S}_{\bi}(t)}$ are the length of 
 ${\cal F}_{\hfIdx}(t)$ and ${\cal S}_{\bi}(t)$, respectively. 

In this study, 
 we employ a staggered Cartesian grid. 
The scalar field $\rho$ is stored as cell-averaged values, 
 while the velocity field $\bu$ is stored as face-averaged values 
 \footnote{Although the advection velocity $\bu$ 
  is given as a known function in this problem, 
  we still stored it as face-averaged values 
  to facilitate future coupling with a main flow solver.}
 (see \cref{fig:staggered_grid}).

The time interval $[0,T]$ is partitioned into $N_{T}$ subintervals 
 with a uniform time step size $k=\frac{T}{N_{T}}$ 
 so that $t_{n}\coloneqq n k$ for $n=0,\ldots,N_{T}$. 

\begin{figure}[htbp]
    \vspace{-15pt}
    \centering
    \begin{tikzpicture}[scale=1.5]
        \filldraw[gray!50]
            (0,0) arc (120:90:4) -- (2,2) -- (0,2) -- cycle;
        \draw[line width=1pt]
            (0,0) arc (120:90:4) -- (2,2) -- (0,2) -- cycle;
        
        \node at (1,1.1) [scale=1.2]
            {$\avg{\rho}_\bi$};

        \filldraw (2,1.268) circle (.05);
        \draw [->] (2,1.268) -- (2.4,1.268);
        \node at (2.4,1.268) [right] 
            {$\avg{u_1}_{\bi+\frac{1}{2}\be^1}$};

        \filldraw (0,1) circle (.05);
        \draw [->] (0,1) -- (0.4,1);
        \node at (0,1) [left] 
            {$\avg{u_1}_{\bi-\frac{1}{2}\be^1}$};

        \filldraw (1,2) circle (.05);
        \draw [->] (1,2) -- (1,2.4);
        \node at (1,2.4) [above] 
            {$\avg{u_2}_{\bi+\frac{1}{2}\be^2}$};
    \end{tikzpicture}
    \caption{A sketch of a cut cell in the staggered grid. 
        $\rho$ is stored as cell-averaged values, 
        and different components of velocity $\bu$ 
        are stored as face-averaged values in different directions.}
    \label{fig:staggered_grid}
    \vspace{-20pt}
\end{figure}
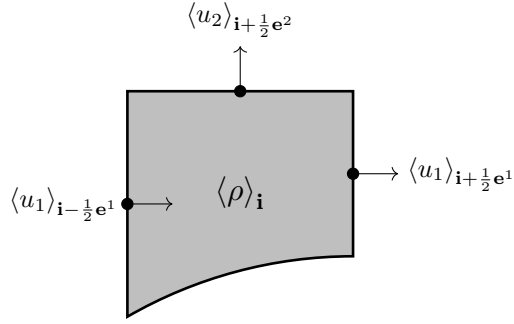

\subsection{Finite volume discretization}
Integrating \cref{eq:AdvDiff} over each cut cell ${\cal C}_{\bi}(t)$ yields
\begin{equation}
    \label{eq:AdvDiff_Integral}
    \iint_{{\cal C}_{\bi}(t)} \frac{\partial \rho}{\partial t}\ \dif \bx = 
     \iint_{{\cal C}_{\bi}(t)} \left(
        -\bu \cdot \nabla\rho + \frac{1}{\mathrm{Pe}}\Delta\rho
     \right)\ \dif \bx.
\end{equation}
Because cut cells are time-dependent, 
 using cell-averaged values as time evolution variables 
 is no longer appropriate. 
Instead, we use cell-integral values $\iint_{{\cal C}_{\bi}(t)}\rho\ \dif \bx$, 
 i.e., $\vol{{\cal C}_{\bi}(t)} \avg{\rho}_{\bi}$ 
 as time evolution variables. 
Applying \cref{thm:RTT} to \cref{eq:AdvDiff_Integral}, 
 we obtain
\begin{equation}
    \label{eq:AdvDiff_Reynolds}
    \begin{split}
    \frac{\dif}{\dif t}\iint_{{\cal C}_{\bi}(t)}\rho\ \dif \bx 
     &= \iint_{{\cal C}_{\bi}(t)}\left(
            -\bu \cdot \nabla\rho + \frac{1}{\mathrm{Pe}}\Delta\rho
        \right)\ \dif \bx + 
     \oint_{\partial{\cal C}_{\bi}(t)}\rho\bar{\bv}\cdot\bn\ \dif l \\
     &= \iint_{{\cal C}_{\bi}(t)}\left(
            -\bu \cdot \nabla\rho + \frac{1}{\mathrm{Pe}}\Delta\rho
        \right)\ \dif \bx + 
     \int_{\partial{\cal S}_{\bi}(t)}\rho\bv\cdot\bn\ \dif l,
    \end{split}
\end{equation}
 where $\bar{\bv}$ is the moving velocity of $\partial {\cal C}_{\bi}(t)$, 
 the second equality holds because $\bar{\bv}\cdot\bn=0$ on cut faces 
 and $\bar{\bv}=\bv$ on the cut boundary. 
Then, the finite volume discretization of \cref{eq:AdvDiff} is given by 
\begin{equation}
    \label{eq:AdvDiff_Finite_Volume}
    \begin{split}
    \frac{\dif}{\dif t}\left(\vol{{\cal C}_{\bi}(t)}\avg{\rho}_{\bi}\right) 
    =& \vol{{\cal C}_{\bi}(t)}\left(
        - \avg{\eadv(\rho,\bu)}_{\bi} 
        + \frac{1}{\mathrm{Pe}}\avg{\elap(\rho)}_{\bi}
     \right) + 
     \mathcal{L}_{\mathrm{net}}(\bv,\rho_{\mathrm{bc}})_{\bi} \\
    \approx& \vol{{\cal C}_{\bi}(t)}\left(
        - \adv\left(\avg{\rho},\avg{\bu}\right)_{\bi} 
        + \frac{1}{\mathrm{Pe}}
         \lap\left(\avg{\rho}\right)_{\bi}
     \right) \\
     &+ \mathbf{L}_{\mathrm{net}}\left(
        \bv, \rho_{\mathrm{bc}}
     \right)_{\bi},
    \end{split}
\end{equation}
 where $\eadv$ is the advection term and 
 $\adv$ computes its fourth-order accurate 
 cell-averaged approximation;
 $\elap$ is the Laplacian term and 
 $\lap$ computes its fourth-order accurate 
 cell-averaged approximation; 
 $\mathcal{L}_{\mathrm{net}}$ is the net transportation term and 
 $\mathbf{L}_{\mathrm{net}}$ computes its approximation by Simpson's rule.

\subsection{Merging algorithm}
\label{sec:Merging_Alg}
As discussed in \Cref{sec:Introduction}, 
 topological changes in cut cells will introduce discontinuities 
 in the time evolution variables. 
In this subsection, 
 we analyze the reason for these discontinuities in more detail 
 and develope an algorithm to address them.

By differentiating \cref{eq:RTT} recursively, 
 we obtain the following corollary. 
 
\begin{corollary}
    \label{cor:RTT_continuity}
    Denote by $\gamma(s,t),\,s\in[0,1]$ the parametric representation 
     of $\partial V(t)$. 
    If both $f$ and $\mathbf{v}$ are smooth, 
     then $\iint_{V(t)}f(\bx, t)\,\dif \bx$ is $C^{\alpha}$ in time 
     if and only if $\gamma(s,t)$ is $C^{\alpha-1}$ in time.
\end{corollary}
\begin{proof}
    First, we differentiate \cref{eq:RTT} again and have 
    \begin{equation}
        \frac{\dif^2}{\dif t^2}\iint_{V(t)}f\,\dif \bx = 
         \iint_{V(t)}\frac{\partial^2 f}{\partial t^2}\,\dif \bx 
         + \oint_{\partial V(t)} 
            \frac{\partial f}{\partial t}\bv\cdot\bn\,\dif l 
         + \frac{\dif}{\dif t}\oint_{\partial V(t)}f\bv\cdot\bn\,\dif l.
    \end{equation}
    So $\iint_{V(t)}f\,\dif \bx$ is $C^{2}$ in time 
     if and only if $\oint_{\partial V(t)}f\bv\cdot\bn\,\dif l$ 
     is $C^{1}$ in time. 
    If we keep differentiating it, 
     we will find that 
     $\iint_{V(t)}f\,\dif \bx$ is $C^\alpha$ in time 
     if and only if $\oint_{\partial V(t)}
     \frac{\partial^{j} f}{\partial t^{j}}\bv\cdot\bn\,\dif l$ 
     is $C^{\alpha-1-j}$ in time for $j=0,\ldots,\alpha-1$.
    
    By the parametric representation of $\partial V(t)$, 
     the integral can be rewritten as 
    \begin{equation}
        \oint_{\partial V(t)}
         \dfrac{\partial^{j} f}{\partial t^{j}}\bv\cdot\bn\,\dif l = 
         \int_{0}^{1} \dfrac{\partial^{j} f}{\partial t^{j}}(\gamma(s,t),t)\,
         \bv(\gamma(s,t),t) \cdot \bn(s,t)
         \dfrac{\partial l}{\partial s}(s,t)\,\dif s,
    \end{equation}
    where $\bn(s,t)$ and $l(s,t)$ are smooth functions of 
     $\frac{\partial \gamma}{\partial s}$. 
    So $\oint_{\partial V(t)}
     \dfrac{\partial^{j} f}{\partial t^{j}}\bv\cdot\bn\,\dif l$ 
     is $C^{\alpha-1-j}$ in time 
     if and only if $\gamma(s,t)$ is $C^{\alpha-1-j}$ in time. 
    Therefore, $\iint_{V(t)}f\,\dif \bx$ is $C^\alpha$ in time 
     if and only if $\gamma(s,t)$ is $C^{\alpha-1}$ in time.
\end{proof}

Since only the cut boundary is contributive in \cref{eq:AdvDiff_Reynolds}, 
 \cref{cor:RTT_continuity} indicates that 
 the discontinuity of cell-integral values 
 lies in the discontinuity of ${\cal S}_{\bi}(t)$. 
If ${\cal S}_{\bi}(t)$ sweeps across a corner node of the cell $\bC_{\bi}$ 
 during a time step as shown in \cref{fig:sweep_corner}, 
 then it is at most $C^0$ in time, 
 thus $\|{\cal C}_{\bi}(t)\|\avg{\rho}_{\bi}$ 
 is at most $C^{1}$ in time, 
 which limits the cut-cell method to second-order accuracy. 

\begin{figure}[htbp]
    \centering
    \begin{tikzpicture}[scale=1.2]
    \draw[line width=1pt] 
        (0,0) rectangle (2,2);
    \filldraw[fill=gray!50]
        (0,0) -- (0.8,0) .. controls (0.5,0.8) .. (0,1.5) -- cycle;
    \draw[line width=2pt]
        (0.8,0) .. controls (0.5,0.8) .. (0,1.5);
    \node at (0.5,0.8) [right]{${\cal S}_{\bi}(t_{n})$};
    \node at (2,1.7) [left]{$\bC_{\bi}$};

    \draw[-{To[length=2mm,width=1.5mm]},line width=1pt]
       (2.5,1.0) -- (3.5,1.0);

    \draw[line width=1pt] 
        (4,0) -- (6,0) -- (6,2) -- (4, 2) -- cycle;
    \filldraw[fill=gray!50]
        (4,0) -- (5.2,0) .. controls (4.8,1) .. (4,2) -- cycle;
    \draw[line width=2pt]
        (5.2,0) .. controls (4.8,1) .. (4,2);
    \node at (4.8,1) [right]{${\cal S}_{\bi}(\tau)$};
    \node at (6,1.7) [left]{$\bC_{\bi}$};

    \draw[-{To[length=2mm,width=1.5mm]},line width=1pt]
       (6.5,1.0) -- (7.5,1.0);

    \draw[line width=1pt] 
        (8,0) -- (10,0) -- (10,2) -- (8, 2) -- cycle;
    \filldraw[fill=gray!50]
        (8,0) -- (9.6,0) .. controls (9,1.2) .. (8.2,2) -- (8,2) -- cycle;
    \draw[line width=2pt]
    (9.6,0) .. controls (9,1.2) .. (8.2,2);
    \node at (9.2,0.8) [left]{${\cal S}_{\bi}(t_{n+1})$};
    \node at (10,1.7) [left]{$\bC_{\bi}$};

    \end{tikzpicture}
    \caption{Example of ${\cal S}_{\bi}(t)$ sweeping a corner of $\bC_{\bi}$, 
     where ${\cal S}_{\bi}(t)$ are the thick lines 
     and ${\cal C}_{\bi}(t)$ are the regions colored by light gray.}
    \label{fig:sweep_corner}
\end{figure}
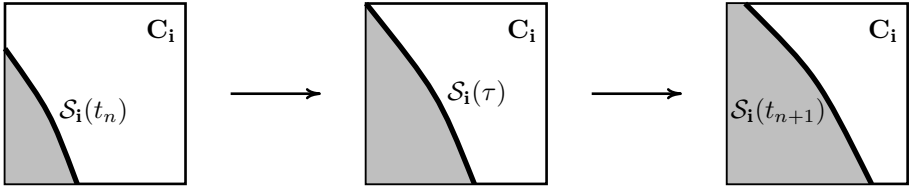

Another situation arises when the boundary enters or exits a cell 
 without sweeping across any corner nodes, 
 as depicted in \cref{fig:bdry_enter}. 
This also creates a discontinuity in $\vol{{\cal C}_{\bi}(t)}\avg{\rho}_{\bi}$. 

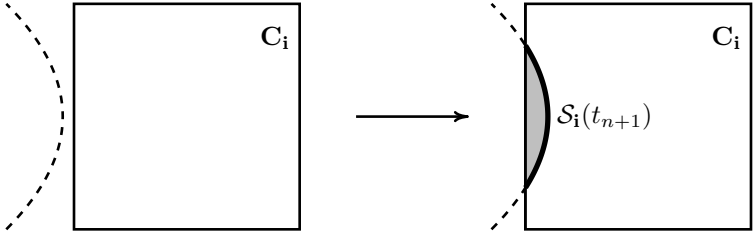
\begin{figure}[htbp]
    \centering
    \begin{tikzpicture}[scale=1.5]
    \draw[line width=1pt] 
        (0,0) -- (2,0) -- (2,2) -- (0, 2) -- cycle;
    \draw[line width=1pt,dashed,domain=0:2,smooth,variable=\y]
        plot ({- 0.1 - 0.5*(\y-1)^2},\y);
    \node at (2,1.7) [left]{$\bC_{\bi}$};

    \draw[-{To[length=2mm,width=1.5mm]},line width=1pt]
       (2.5,1.0) -- (3.5,1.0);

    \draw[line width=1pt] 
        (4,0) -- (6,0) -- (6,2) -- (4, 2) -- cycle;
    \draw[name path = A, line width=1pt,dashed,domain=0:2,smooth,variable=\y]
        plot ({4.2 - 0.5*(\y-1)^2},\y);
    \draw[line width=2pt,domain=0.37:1.63,smooth,variable=\y]
        plot ({4.2 - 0.5*(\y-1)^2},\y);
    \draw[name path = B, line width=1pt] 
        (4,0) -- (4,2);
    \begin{pgfonlayer}{bg}
        \fill [gray!50,
            intersection segments={
                of=A and B,
                sequence={L2--R2}
            }];
    \end{pgfonlayer}
    \node at (4.2,1) [right]{${\cal S}_{\bi}(t_{n+1})$};
    \node at (6,1.7) [left]{$\bC_{\bi}$};

    \end{tikzpicture}
    \caption{Example of the boundary entering a cell $\bC_{\bi}$ in a time step 
     without sweeping across any corner nodes.}
    \label{fig:bdry_enter}
\end{figure}

In this work, 
 we introduce a merging technique to address these problems 
 caused by boundaries movement. 
The key issue lies in properly handling nodes and faces 
 whose attributes change during a time step 
 (from being inside the fluid region to being outside the fluid region 
 or vice-versa). 

\begin{enumerate}[({TMA}-1)]
    \item For each node $\bN_{\bi}$ whose attribute changes 
     during the time step $[t_{n},t_{n+1}]$, 
     denote by $t=t_{n}\ \mathrm{or}\ t_{n+1}$ the time 
     at which $\bN_{\bi}$ is an interior node. 
    We first find the smallest cut face adjacent to $\bN_{\bi}$, 
     for example, ${\cal F}_{\bi-\frac{1}{2}\be^{1}-\be^{2}}(t)$. 
    We then merge $\bF_{\bi-\frac{1}{2}\be^{1}-\be^{2}}$ 
     and $\bF_{\bi-\frac{1}{2}\be^{1}}$ to form a merged face, 
     the corresponding cut faces at $t_{n}$ and $t_{n+1}$ are also merged. 
    In addition, 
     the cells on both sides of the merged face also need to be merged 
     (merge $\bC_{\bi}$ and $\bC_{\bi-\be^{2}}$, 
      merge $\bC_{\bi-\be^{1}}$ and $\bC_{\bi-\mathbbm{1}}$),
     as shown in \cref{fig:merge_sweep_corner}. 
    This process prevent $\bN_{\bi}$ from being a corner 
     of any merged cut cell, 
     thereby addressing the situation in \cref{fig:sweep_corner};
    \item For any unmerged faces $\bF_{\hfIdx}$ whose attribute changes 
     during the time step $[t_n,t_{n+1}]$, 
     we simply merge the cells $\bC_{\bi}$ and $\bC_{\bi+\be^{d}}$ 
     to form a merged cell, as depicted in \cref{fig:merge_bdry_enter}. 
    This ensures that $\bF_{\hfIdx}$ is no longer a face 
     of any merged cut cell, 
     addressing the situation in \cref{fig:bdry_enter}.
\end{enumerate}

\begin{figure}[htbp]
    \centering
    \begin{tikzpicture}[scale=0.95]

    \coordinate (A1) at (-1.8, -1.2);
    \coordinate (B1) at (-1, -0.2);
    \coordinate (C1) at (1, 1.2);
    \coordinate (D1) at (2.2, 1.8);

    \filldraw[gray!50] 
        (A1) plot[smooth, tension = 1] coordinates 
        {(A1) (B1) (C1) (D1)} -- (D1) -- (2.2,2.2) 
        -- (-2.2,2.2) -- (-2.2,-1.2) -- cycle;

    \draw[help lines, black, line width = .1pt] 
        (-2.2, -1.2) grid (2.2, 2.2);

    \begin{scope}
        \clip
            (A1) plot[smooth, tension = 1] coordinates 
            {(A1) (B1) (C1) (D1)} -- (D1) -- (2.2,2.2) 
            -- (-3.2,2.2) -- (-3.2,-2.2) -- cycle;
        \filldraw[fill=gray, line width=.1pt] 
            (0,0) rectangle (1,2);
    \end{scope}

    \begin{scope}
        \clip
            (A1) plot[smooth, tension = 1] coordinates 
            {(A1) (B1) (C1) (D1)} -- (D1) -- (2.2,2.2) 
            -- (-3.2,2.2) -- (-3.2,-2.2) -- cycle;
        \filldraw[fill=gray, line width=.1pt] 
            (1,0) rectangle (2,2);
    \end{scope}

    \node at (1,1) [rectangle, draw=black, fill=black, 
        inner sep=0pt, minimum size=3pt] {};
    \node at (1,0.7) [right] {$\bN_{\bi}$};

    \draw[line width=1pt] 
        (A1) plot[smooth, tension = 1] coordinates 
        {(A1) (B1) (C1) (D1)};
    \node at (D1) [right] {$\partial\Omega(t_n)$};

    \draw[-{To[length=2mm,width=1.5mm]},line width=1pt]
        (2.5,0.5) -- (3.5,0.5);

    \coordinate (A2) at (4.8, -1.2);
    \coordinate (B2) at (5.5, -0.2);
    \coordinate (C2) at (7, 0.8);
    \coordinate (D2) at (8.2, 1.2);

    \filldraw[gray!50] 
        (A2) plot[smooth, tension = 1] coordinates 
        {(A2) (B2) (C2) (D2)} -- (D2) -- (8.2,2.2) 
        -- (3.8,2.2) -- (3.8,-1.2) -- cycle;

    \draw[help lines, black, line width = .1pt] 
        (3.8, -1.2) grid (8.2, 2.2);

    \begin{scope}
        \clip
            (A2) plot[smooth, tension = 1] coordinates 
            {(A2) (B2) (C2) (D2)} -- (D2) -- (8.2,2.2) 
            -- (3.8,2.2) -- (3.8,-1.2) -- cycle;
        \filldraw[fill=gray, line width=.1pt] 
            (6,0) rectangle (7,2);
    \end{scope}

    \begin{scope}
        \clip
            (A2) plot[smooth, tension = 1] coordinates 
            {(A2) (B2) (C2) (D2)} -- (D2) -- (8.2,2.2) 
            -- (3.8,2.2) -- (3.8,-1.2) -- cycle;
        \filldraw[fill=gray, line width=.1pt] 
            (7,0) rectangle (8,2);
    \end{scope}

    \node at (7,1) [rectangle, draw=black, fill=black, 
        inner sep=0pt, minimum size=3pt] {};

    \draw[line width=1pt] 
        (A2) plot[smooth, tension = 1] coordinates 
        {(A2) (B2) (C2) (D2)};
    \node at (D2) [right] {$\partial\Omega(t_{n+1})$};
    \end{tikzpicture}
    \vspace{-30pt}
    \caption{The attribute of the node $\bN_{\bi}$ 
     changes during the time step $[t_n,t_{n+1}]$, 
     so the adjacent cells are merged to form the dark gray cells 
     and the node is no longer a corner node of both merged cells.}
    \label{fig:merge_sweep_corner}
\end{figure}
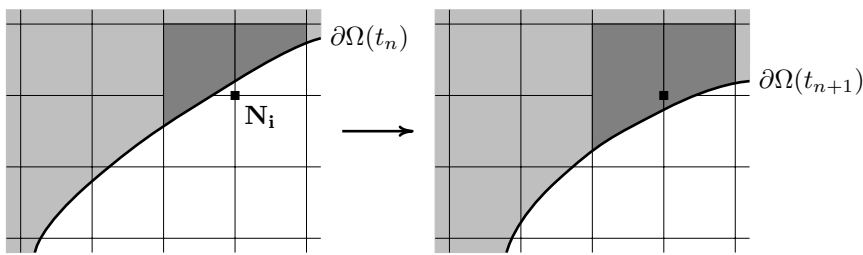

\begin{figure}[htbp]
    \centering
    \begin{tikzpicture}[scale=1.5]

    \filldraw[gray!50] 
        (-2.6,-1.2) arc (-30:30:3.4) -- (-3.2,2.2) -- (-3.2,-1.2) 
        -- cycle;

    \draw[help lines, black, line width = .1pt] 
        (-3.2, -1.2) grid (-0.8, 2.2);

    \begin{scope}
        \clip
            (-2.6,-1.2) arc (-30:30:3.4) -- (-3.2,2.2) -- (-3.2,-1.2) 
            -- cycle;
        \filldraw[fill=gray, line width=.1pt] 
            (-3,0) rectangle (-1,1);
    \end{scope}
    \node at (-2,0.5) [right] {$\bF_{\bi+\frac{1}{2}\be^1}$};

    \draw[line width=.5pt] 
        (-2.6,-1.2) arc (-30:30:3.4);
    \node at (-2.6,2.2) [above] {$\partial\Omega(t_n)$};

    \draw[-{To[length=2mm,width=1.5mm]},line width=1pt]
        (-0.5,0.5) -- (0.5,0.5);

    \filldraw[gray!50] 
        (1.575,-1.2) arc (-30:30:3.4) -- (0.8,2.2) -- (0.8,-1.2) 
        -- cycle;

    \draw[help lines, black, line width = .1pt] 
        (0.8, -1.2) grid (3.2, 2.2);

    \begin{scope}
        \clip
            (1.575,-1.2) arc (-30:30:3.4) -- (0.8,2.2) -- (0.8,-1.2) 
            -- cycle;
        \filldraw[fill=gray, line width=.1pt] 
            (1,0) rectangle (3,1);
    \end{scope}

    \draw[line width=.5pt] 
        (1.575,-1.2) arc (-30:30:3.4);
    \node at (1.575,2.2) [above] {$\partial\Omega(t_{n+1})$};

    \end{tikzpicture}
    \vspace{-10pt}
    \caption{The attribute of the face $\bF_{\bi+\frac{1}{2}be^1}$ 
     changes during the time step $[t_n,t_{n+1}]$, 
     so that adjacent cells are merged to form the dark gray cell 
     and the face is no longer a face of any cell.}
    \label{fig:merge_bdry_enter}
\end{figure}
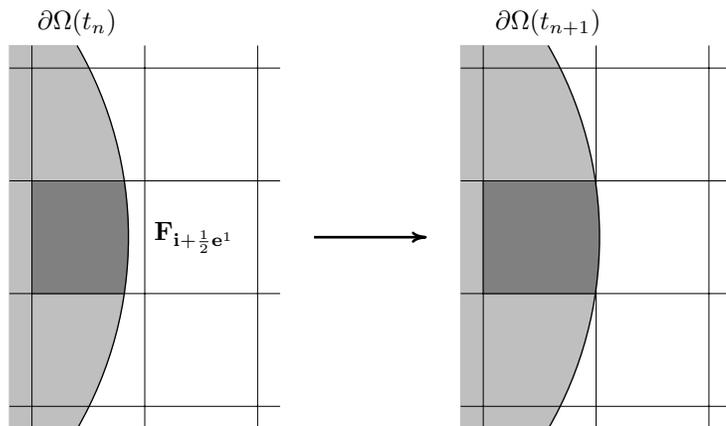

This cell-merging technique ensure that 
 corner nodes and faces undergoing attribute changes no longer exist, 
 so the temporal continuities of $\vol{{\cal C}_{\bi}(t)}\avg{\rho}_{\bi}$ 
 will not be disrupted.

Additionally, 
 cut cell or face can be arbitrarily small in practice, 
 leading to ill-conditioning operators. 
Furthermore, the CFL stability conditions require a very small time step size 
 for these small cells. 
To address these problems, 
 our merging technique also processes these 
 small cells and faces through the following steps:
\begin{enumerate}[({SMA}-1)]
    \item For any unmerged cell $\bC_{\bi}$ satisfying 
     $\vol{{\cal C}_{\bi}(t_{n})},\vol{{\cal C}_{\bi}(t_{n+1})} > 0$ and 
     $\vol{{\cal C}_{\bi}(t_{n})}\ \mathrm{or}\ \vol{{\cal C}_{\bi}(t_{n+1})}
      \leq \epsilon_{c} h^2$, 
     we locate the adjacent cell $\bC_{\bj}$ 
     that shares the longest common cut face with $\bC_{\bi}$ 
     at time $t_{n}\ \mathrm{or}\ t_{n+1}$, 
     and merge $\bC_{\bi}$ and $\bC_{\bj}$ to form a merged cell; 
    \item For any unmerged face $\bF_{\hfIdx}$ satisfying 
     $\vol{{\cal F}_{\hfIdx}(t_{n})},\vol{{\cal F}_{\hfIdx}(t_{n+1})}>0$ and 
     $\vol{{\cal F}_{\hfIdx}(t_{n})}\ \mathrm{or} \vol{{\cal F}_{\hfIdx}(t_{n+1})}
      \leq\epsilon_{f} h$. 
    We first attempt to merge it with an adjacent face. 
    Additionally, if two faces are merged, 
     the cells on both sides of the merged face also need to be merged.
    If no suitable adjacent face is found, 
     we then simply merge the two cells on both sides of $\bF_{\hfIdx}$.
\end{enumerate}
Here, $\epsilon_{c}$ and $\epsilon_{f}$ are user-defined parameters, 
 which are usually set to be $\epsilon=\epsilon_{c}=\epsilon_{f}\in[0.1, 0.2]$.

\subsection{Discrete operators}
Benefiting from the Cartesian grid, 
 we can apply standard discretization formula 
 to compute $\lap$ and $\adv$ 
 on cut cells far from boundaries. 
For cut cells near boundaries, 
 the standard discretization formula is no longer applicable, 
 since part of the cut cells in the stencil are interface cells or empty cells. 
In this case, 
 we employ the PLG algorithm~\cite{Zhang2024PLG} 
 to generate a stencil that conforms to the local geometry, 
 and use polynomial interpolation technique 
 to obtain the discretization formula. 

\subsubsection{Discrete Laplacian operator}
The fourth-order standard discretization formula 
 for $\lap$ is 
\begin{equation}
    \label{eq:Regular_lap}
    \begin{split}
    \lap\left(\avg{\rho}\right)_{\bi} 
    &\coloneqq \frac{1}{12h^{2}}\sum_{d}\left( 
        \avg{\rho}_{\bi+2\be^{d}} + 16\avg{\rho}_{\bi+\be^{d}} 
        - 30\avg{\rho}_{\bi} + 16\avg{\rho}_{\bi-\be^{d}} 
        + \avg{\rho}_{\bi-2\be^{d}}
     \right) \\
    &= \avg{\elap(\rho)}_\bi + O(h^4),
    \end{split}
\end{equation}
 where the $O(h^4)$ error is by Taylor expansion. 

For cut cell ${\cal C}_{\bi}$ near boundaries, we first introduce the LIP. 
\begin{definition}[Lagrange interpolation problem (LIP)]
    Denote by $\Pi_{q}$ the linear space of all bi-variate polynomials 
     of degree no greater than $q$ with real coefficients 
     and write $N\coloneqq \mathrm{dim}\Pi_{q}$. 
    For a finite set of pairwise disjoint cut cells 
     $\mathcal{T} \coloneqq \{{\cal C}_{\bi_{1}},\ldots,{\cal C}_{\bi_{N}}\}$ 
     and a corresponding set of cell-averaged values 
     $f_{1},\ldots,f_{N}\in\mathbb{R}$. 
    The \emph{LIP} seeks a polynomial $f\in\Pi_{q}$ such that 
    \begin{equation}
        \label{eq:LIP}
        \forall j=1,\ldots,N, \quad \avg{f}_{\bi_{j}}=f_{j}.
    \end{equation}
\end{definition}

A LIP is \emph{unisolvent} if for \emph{any} given sequence 
 of cell-averaged values $(f_{j})_{j=1}^{N}$ 
 there exists $f\in\Pi_{q}$ satisfying \cref{eq:LIP}; 
 then we also call the set $\mathcal{T}$ a \emph{poised lattice}. 

In this work, we employ an AI-aided algorithm in~\cite{Zhang2024PLG} 
 to find a poised lattice for ${\cal C}_{\bi}$ 
 that conforms to the local geometry. 
The lattice should ideally be centered at ${\cal C}_{\bi}$ 
 to align with the physical characteristics of the Laplacian operator, 
 see \cref{fig:Stencil4Lap} for an example. 
Additionally, for an interface cell, 
 its cut boundary ${\cal S}_{\bi}$ is also added into the generated lattice 
 to incorporate the boundary condition. 

\begin{figure}[htbp]
    \vspace{-10pt}
    \centering
    \includegraphics[width = 1.0\linewidth]{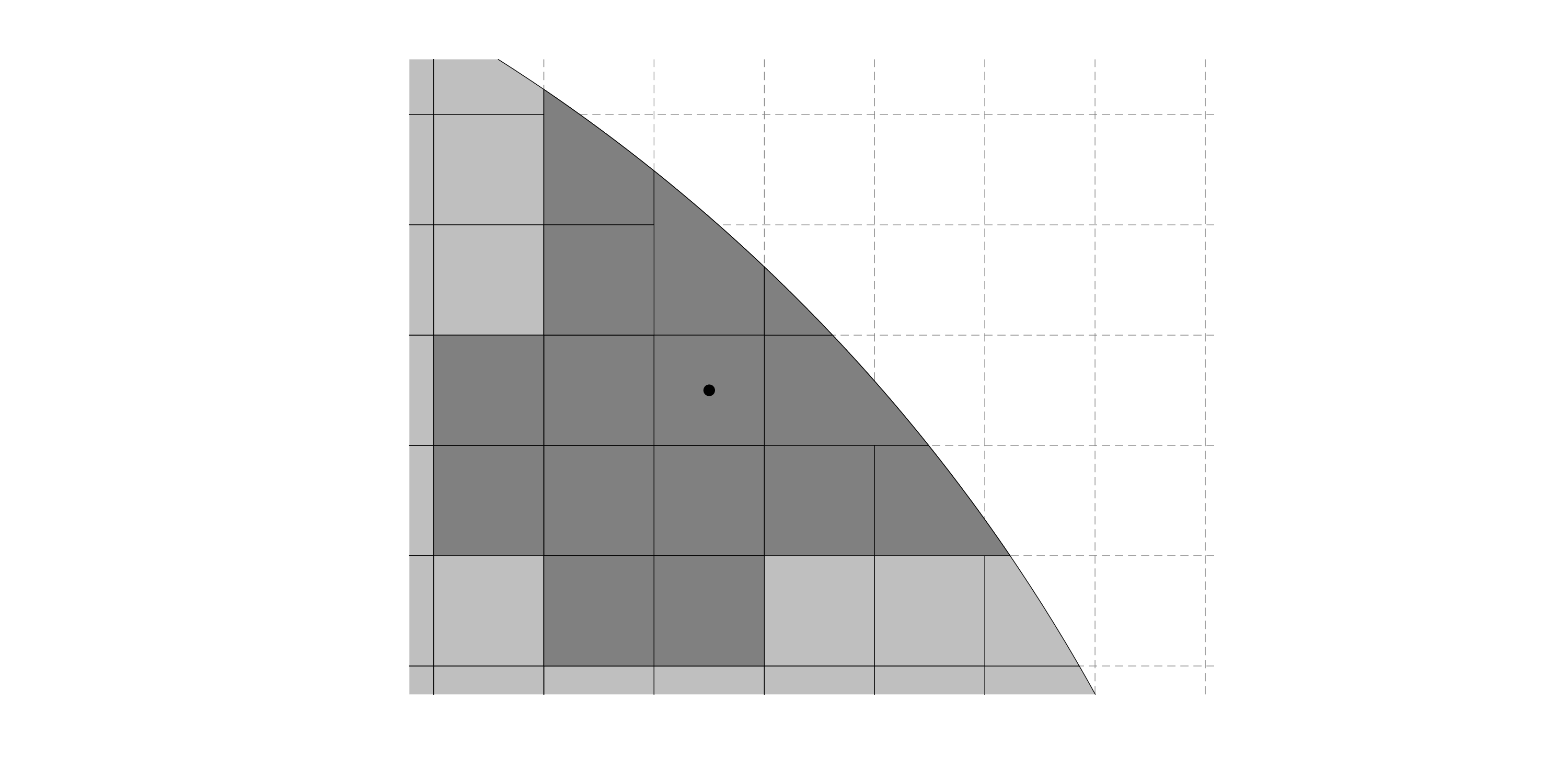}
    \vspace{-30pt}
    \caption{An example of poised lattice for $\lap$. 
     The target cut cell ${\cal C}_{\bi}$ is marked by a black point 
     and the lattice consists of cut cells colored by dark gray.}
    \vspace{-15pt}
    \label{fig:Stencil4Lap}
\end{figure}

For an interface cell ${\cal C}_{\bi}$, we define 
\begin{equation}
    \label{eq:bold_rho}
    \boldsymbol{\rho} = \left[
        \avg{\rho}_{\bi_{1}}, \ldots, \avg{\rho}_{\bi_{N}}, \bbrk{\rho}_{\bi}
     \right]^{\top} \in \bbR^{N+1}.
\end{equation}
After finding a poised lattice, 
 traditional way is to solve the LIP 
 and get a locally $q$th-order polynomial by $\boldsymbol{\rho}$, 
 therefore $\lap(\avg{\rho})_{\bi}$ can be calculated 
 by imposing $\elap$ to the polynomial 
 and computing the cell-averaged value on ${\cal C}_{\bi}$. 
Here, however, 
 we directly calculate a linear combination of $\boldsymbol{\rho}$ 
 as $\lap(\avg{\rho})_{\bi}$, i.e., 
\begin{equation}
    \label{eq:PLG4Lap}
    \avg{\elap(\rho)}_{\bi} 
     = \boldsymbol{\beta}^{\top}\boldsymbol{\rho} + O(h^{q-1})
     = \sum_{j=1}^{N}\beta_{j}\avg{\rho}_{\bi_{j}} + \beta_{b}\bbrk{\rho}_{\bi}
       + O(h^{q-1})
\end{equation}
 holds\footnote{
    The error $O(h^{q-1})$ in \cref{eq:PLG4Lap} arises 
    because we aim to fit a $q$th order polynomial, 
    and $\elap$ is a second order differential operator.}
 for any sufficiently smooth function $\rho:\bbR^{2}\to\bbR$, 
 where $\boldsymbol{\beta}=[\beta_{q},\ldots,\beta_{N},\beta_{\mathrm{b}}]$ 
 is the set of coefficients, 
 which requires that \cref{eq:PLG4Lap} holds exactly for all $\rho\in\Pi_{q}$. 
Denote by $\{\psi_{i}\}_{i=1}^{N}$ the monomial bases of $\Pi_{q}$, 
 the requirement is equivalent to 
\begin{equation*}
    \avg{\elap(\psi_{i})}_{\bi} 
     = \sum_{j=1}^{N} \beta_{j}\avg{\psi_{i}}_{\bi_{j}} 
       + \beta_\mathrm{b} \bbrk{\psi_{i}}_{\bi}
\end{equation*}
 holding for $i=1,\ldots,N$, or more concisely, 
 $M\boldsymbol{\beta} = \boldsymbol{L}$, where 
\begin{equation}
    \label{eq:sampleMat_Lap}
    M = 
    \begin{pmatrix}
        \avg{\psi_{1}}_{\bi_{1}} & \ldots & \avg{\psi_{1}}_{\bi_{N}} & 
         \bbrk{\psi_{1}}_{\bi} \\
        \vdots & \ddots & \vdots & \vdots \\
        \avg{\psi_{N}}_{\bi_{1}} & \ldots & \avg{\psi_{N}}_{\bi_{N}} & 
         \bbrk{\psi_{N}}_{\bi}
    \end{pmatrix}
    \in \bbR^{N\times(N+1)},
\end{equation}
and
\begin{equation*}
    \boldsymbol{L} = \left[
        \avg{\elap(\psi_{1})}_{\bi}, \ldots, 
        \avg{\elap(\psi_{N})}_{\bi}
     \right]^{\top}
     \in \bbR^{N}.
\end{equation*}
This system for $\boldsymbol{\beta}$ is underdetermined 
 and the matrix $M$ has full row rank. 
Therefore, we aim to find the weighted minimum norm solution of 
\begin{equation*}
    \min_{\boldsymbol{\beta}\in\bbR^{N+1}}
     \|\boldsymbol{\beta}\|_{W} \quad 
     \mathrm{s.t.} \quad M\boldsymbol{\beta}=\boldsymbol{L},
\end{equation*}
 where $W=\mathrm{diag}\left(w_{1},\ldots,w_{N},w_{\mathrm{b}}\right)$ 
 and 
\begin{equation}
    \label{eq:WeightFunction}
    w_{j} = \min\{\|\bi_{j}-\bi\|_{2}^{-1},w_{\max}\}, \quad 
    w_{\mathrm{b}} = w_{\max},
\end{equation}
 with $w_{\max}=2$ to avoid large weights. 

For pure cells, the last column of $M$ in \cref{eq:sampleMat_Lap} is omitted, 
 and the square linear system $M\boldsymbol{\beta}=\boldsymbol{L}$ 
 is solved directly for the coefficients. 

\subsubsection{Discrete advection operator}
For cut cells far from boundaries, 
 we use a conservative approach for $\adv$. 
Since the velocity field $\bu$ is incompressible, 
 the advection term can be expressed as 
 $\bu\cdot\nabla\rho = \nabla\cdot(\rho\bu)$.
Applying the divergence theorem, 
 the fourth-order discretization formula for $\adv$ is
\begin{equation}
    \label{eq:Regular_adv}
    \begin{split}
    \adv\left(\avg{\rho},\avg{\bu}\right)_{\bi} 
    &\coloneqq \dfrac{1}{h}\sum_{d,\pm}\left(
        \pm\avg{\rho}_{\bi\pm\frac{1}{2}\be^d}
        \avg{u_d}_{\bi\pm\frac{1}{2}\be^d} \pm 
        \dfrac{h^2}{12}\sum_{d'\neq d}
        \left(\mathbf{G}_{d'}^\perp\rho\right)_{\bi\pm\frac{1}{2}\be^d}
        \left(\mathbf{G}_{d'}^\perp u_d\right)_{\bi\pm\frac{1}{2}\be^d}
     \right) \\ 
    &= \avg{\mathcal{L}_\mathrm{adv}(\rho,\bu)}_\bi + O(h^4), 
    \end{split}
\end{equation}
where
\begin{equation*}
    \begin{split}
    \avg{\rho}_{\bi\pm\frac{1}{2}\be^{d}} 
    &= \frac{7}{12}\left(
        \avg{\rho}_{\bi} + \avg{\rho}_{\bi\pm\be^{d}}
     \right) - \frac{1}{12}\left(
        \avg{\rho}_{\bi\mp\be^{d}} + \avg{\rho}_{\bi\pm 2\be^{d}}
     \right) + O(h^4),\\
    \left(\mathbf{G}_{d'}^\perp g\right)_{\bi\pm\frac{1}{2}\be^d} 
    &= \frac{1}{2h}\left(
        \avg{g}_{\bi\pm\frac{1}{2}\be^d+\be^{d'}} - 
        \avg{g}_{\bi\pm\frac{1}{2}\be^d-\be^{d'}}
     \right)\ \mathrm{for}\ g = \rho\ \mathrm{or} \ u_{d}.
    \end{split}
\end{equation*}
For a detailed derivation, refer to \cite[Appendix B]{Zhang2012AdvDiff}.

For cut cell ${\cal C}_{\bi}$ near boundaries, 
 denote by $\eadv^{1}\coloneqq 
    \frac{\partial (u_{1}\rho)}{\partial x}$ and 
 $\eadv^{2}\coloneqq 
    \frac{\partial (u_{2}\rho)}{\partial y}$, 
 we approximate the cell-averaged value of 
 $\eadv^{1}$ and $\eadv^{2}$ 
 separately. 
Consider $\eadv^{1}$ as an example, 
 observing that it is a bilinear operator, 
 we aim to construct a bilinear form 
 to approximate the cell-averaged value. 

We first employ the algorithm in~\cite{Zhang2024PLG} 
 to generate two stencils for ${\cal C}_{\bi}$, 
 one consists of cut cells $({\cal C}_{\bi_{j}})_{j=1}^{N}$ (for $\rho$), 
 and the other consists of cut faces 
 $({\cal F}_{\bk_{j}+\frac{1}{2}\be^{1}})_{j=1}^{N}$ (for $u_{1}$). 
Because of the physical characteristics of the advection operator, 
 the poised lattice is biased towards the upwind direction. 
Define 
\begin{equation}
    \label{eq:UpwindLoc}
    \bp_{1} \coloneqq \mathbf{p} - \left\lfloor\frac{q}{2}\right\rfloor 
     h \frac{\bu}{\|\bu\|_{2}}, \quad 
    \bp_{2} \coloneqq \mathbf{p} - \left\lfloor\frac{q}{2}\right\rfloor 
    h \bn_{\mathrm{b}},
\end{equation}
 where $\mathbf{p}$ is the cell center of ${\cal C}_{\bi}$ 
 and $\bn_{\mathrm{b}}$ is the unit outward normal 
 at the closest point on $\partial\Omega$ to $\mathbf{p}$. 
As a particle travels from $\bp_{1}$ to $\bp_{2}$, 
 it intersects a sequence of cells, 
 then we attempt to find a poised lattice centered at each cell one by one.  
See \cref{fig:Stencil4Adv} for an example. 
Additionally, if ${\cal C}_{\bi}$ is an interface cell, 
 the cut boundary ${\cal S}_{\bi}$ is added to both lattices 
 to incorporate the boundary condition. 

\begin{figure}[htbp]
    \vspace{-10pt}
    \centering
    \includegraphics[width = 1.0\linewidth]{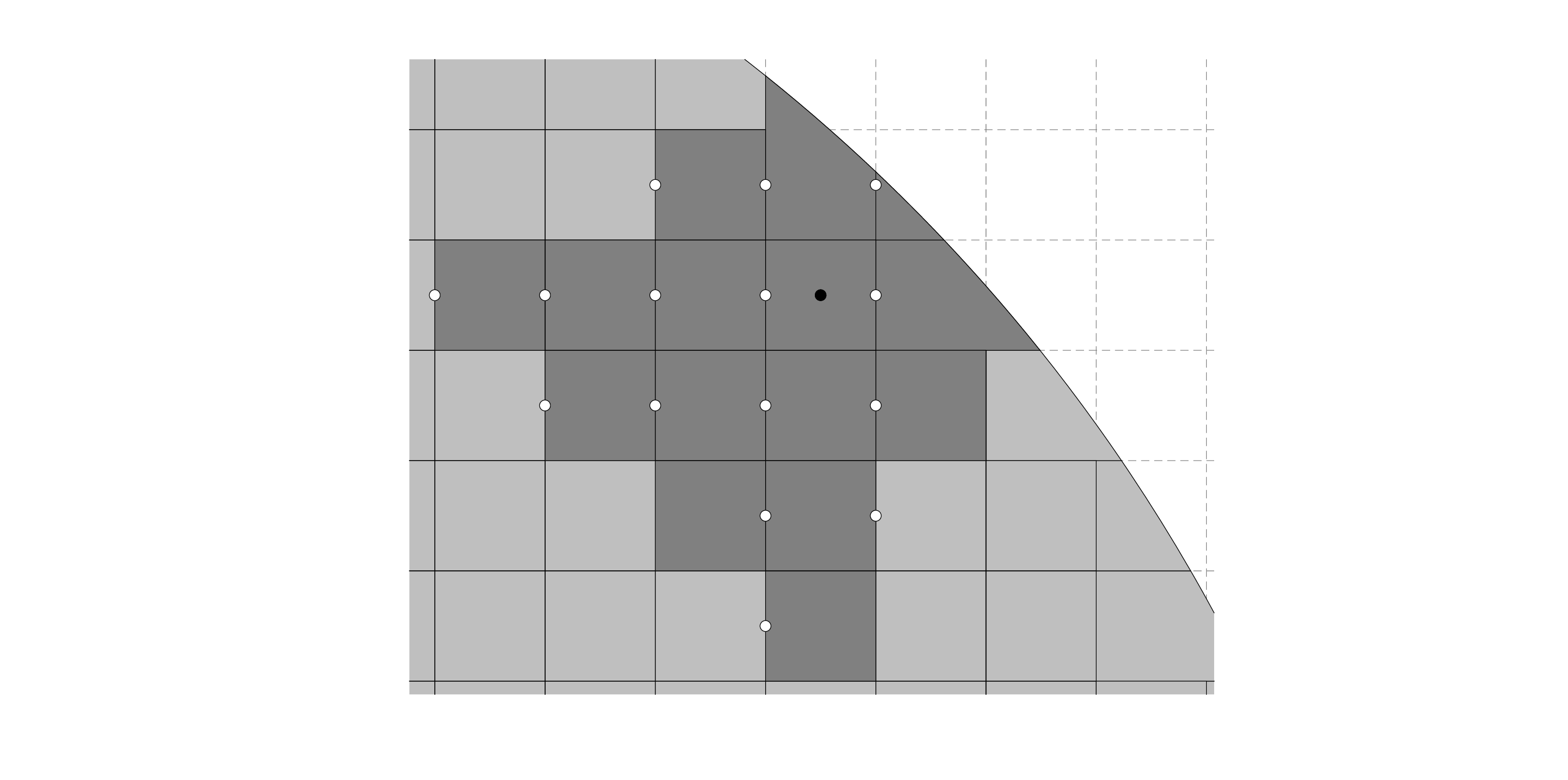}
    \vspace{-30pt}
    \caption{An example of poised lattices for 
     $\eadv^{1}$ with $\bu=(1,1)^{\top}$. 
     The target cut cell ${\cal C}_\bi$ is marked by a black point.
     $({\cal C}_{\bi_{j}})_{j=1}^{N}$ are cut cells colored by dark gray, 
     and $({\cal F}_{\bk_{j}+\frac{1}{2}\be^{1}})_{j=1}^{N}$ 
     are cut faces marked by white circles.}
    \label{fig:Stencil4Adv}
    \vspace{-15pt}
\end{figure}

For an interface cell ${\cal C}_{\bi}$, 
 we define $\boldsymbol{\rho}$ as in \cref{eq:bold_rho} and 
\begin{equation*}
    \boldsymbol{u_{1}} = \left[
        \avg{u_{1}}_{\bk_{1}+\frac{1}{2}\be^{1}}, \ldots, 
        \avg{u_{1}}_{\bk_{N}+\frac{1}{2}\be^{1}}, 
        \bbrk{u_{1}}_{\bi}
     \right]^\top \in \bbR^{N+1}.
\end{equation*}
We seek a matrix $A = (A_{j_{1} j_{2}})\in \bbR^{(N+1)\times(N+1)}$
 such that the bilinear approximation formula
\begin{equation}
    \label{eq:PLG4Adv}
    \avg{\eadv^1(\rho,\bu)}_{\bi} = 
     \boldsymbol{\rho}^{\top}A \boldsymbol{u_{1}} + O(h^q)
\end{equation}
 holds for all sufficiently smooth function $\rho:\bbR^2\to \bbR$ 
 and $u_{1}:\bbR^2\to \bbR$.
This requires that \cref{eq:PLG4Adv} holds exactly 
 for all $\rho\in\Pi_{q}$ and $u_{1}\in\Pi_{q}$, 
 which is equivalent to 
\begin{equation}
    \label{eq:PLG4AdvMat}
    M_{1}AM_{2}^{\top}=\hat{\boldsymbol{L}},
\end{equation}
 where 
\begin{equation*}
    M_1=M; \quad 
    M_2 = 
    \begin{pmatrix}
    \avg{\psi_{1}}_{\bk_{1}+\frac{1}{2}\be^1} & \cdots & 
     \avg{\psi_{1}}_{\bk_{N}+\frac{1}{2}\be^1} & 
     \bbrk{\psi_{1}}_{\bi} \\
    \vdots & \ddots & \vdots & \vdots \\
    \avg{\psi_{N}}_{\bk_{1}+\frac{1}{2}\be^1} & \cdots & 
     \avg{\psi_{N}}_{\bk_{N}+\frac{1}{2}\be^1} & 
     \bbrk{\psi_{N}}_{\bi}
    \end{pmatrix}
    \in \bbR^{N\times(N+1)},
\end{equation*}
and 
\begin{equation*}
    \hat{\boldsymbol{L}} = 
    \begin{pmatrix}
    \avg{\eadv^{1}(\psi_{1},\psi_{1})}_{\bi} & \cdots & 
     \avg{\eadv^{1}(\psi_{1},\psi_{N})}_{\bi} \\ 
    \vdots & \ddots & \vdots \\
    \avg{\eadv^{1}(\psi_{N},\psi_{1})}_{\bi} & \cdots & 
     \avg{\eadv^{1}(\psi_{N},\psi_{N})}_{\bi}
    \end{pmatrix}
    \in \bbR^{N\times N}.
\end{equation*}

To solve \cref{eq:PLG4AdvMat}, we first let $A_{1}=AM_{2}^{\top}$. 
Given that $M_{1}$ has full row rank, 
 we can compute $A_{1}$ by solving $M_{1}A_{1}=\hat{\boldsymbol{L}}$ 
 column by column, 
 using a weighted minimum norm approach. 
Once $A_{1}$ is determined, 
 we can compute $A$ by solving $M_{2}A^{\top}=A_{1}^{\top}$ column by column, 
 using a weighted minimum norm approach. 
The weight matrix is chosen consistently with \cref{eq:WeightFunction}.

For pure cells, the last columns of $M_1$ and $M_2$ are omitted, 
 then $M_1A_1=\hat{\mathbf{L}}$ and $M_2A^{\top}=A_1^{\top}$ 
 are square linear systems which can be solved directly.

\subsection{Time integration}
To solve the ODEs system \cref{eq:AdvDiff_Finite_Volume}, 
 we use the implicit-explicit Runge-Kutta (IMEX) scheme 
 introduced by Kennedy and Carpenter~\cite{Kennedy2003IMEX}
 for time integration. 
We implement the ARK4(3)6L[2]SA method, 
 which is a six-stage, fourth-order accurate, and L-stable scheme. 
In this work, 
 We treat the advection term 
 $-\vol{{\cal C}_{\bi}(t)}\adv(\avg{\rho},\avg{\bu})_{\bi}$ 
 explicitly, 
 while handling the Laplacian term 
 $\frac{\vol{{\cal C}_{\bi}(t)}}{\mathrm{Pe}}
  \lap(\avg{\rho})_{\bi}$ 
 and the net transportation term 
 $\mathbf{L}_{\mathrm{net}}(\bv,\rho_{\mathrm{bc}})_{\bi}$ 
 implicitly. 

Given the moving boundary 
 $\partial\Omega^{(1)}\coloneqq\partial\Omega(t_{n})$ at time $t_{n}$, 
 the velocity field $\bv$, 
 and the initial cell-averaged values 
 $\avg{\rho}_{\bi}^{(1)}\coloneqq\avg{\rho(t)}_{\bi}^{n}$, 
 the time step $[t_{n},t_{n+1}]$ proceeds as follows: 
\begin{enumerate}[({TI}-1)]
    \item Employ the ARMS technique~\cite{Hu2024ARMS} 
     to obtain the moving boundary 
     $\partial\Omega^{(s)}\coloneqq\partial\Omega(t_{n}+c_{s}k)$ 
     at each stage, $s=2,\ldots,6$. 
    We then embed all $\Omega^{(s)}$ into $\Omega_{R}$, 
     forming the corresponding ${\cal C}_{\bi}^{(s)}$, 
     ${\cal F}_{\hfIdx}^{(s)}$ and ${\cal S}_{\bi}^{(s)}$;
    \item Apply TMA and SMA procedures introduced in \Cref{sec:Merging_Alg}
     to merge cells and faces, 
     ensuring the continuity of $\vol{{\cal C}_{\bi}(t)}\avg{\rho}_{\bi}$ 
     during $[t_n,t_{n+1}]$ on each merged cut cells. 
    Additionally, this ensures that 
     no merged cut cells or faces are smaller than the user-defined threshold, 
     reducing the condition number of operators 
     and alleviate the CFL constraint;
    \item For $s=2,\ldots,6$, we need to solve Helmholtz-type equations: 
    \begin{equation}
        \label{eq:Helmholtz_type}
        \left(I-\frac{\gamma k}{\mathrm{Pe}} \lap\right)
            \avg{\rho}_{\bi}^{(s)} = 
         \frac{1}{\vol{{\cal C}_{\bi}^{(s)}}} \left(
            \vol{{\cal C}_{\bi}^{(1)}}\avg{\rho}_{\bi}^{(1)} + k \hat{L}
         \right)
    \end{equation}
     to get $\avg{\rho}_{\bi}^{(s)}$ on each merged cut cell, 
     where $\gamma=a_{s,s}^{[I]}$ is a constant for all stages, 
     and 
    \begin{equation*}
        \begin{split}
        \hat{L} =
        &-\sum_{j=1}^{s-1} a_{s,j}^{[E]}
           \vol{{\cal C}_{\bi}^{(j)}} \adv^{(j)}\left(
                \avg{\rho}^{(j)}, \avg{\bu}^{(j)}
            \right)_{\bi}
         +\frac{1}{\mathrm{Pe}} \sum_{j=1}^{s-1} a_{s,j}^{[I]} 
            \vol{{\cal C}_{\bi}^{(j)}} \lap^{(j)}\left(
                \avg{\rho}^{(j)}
            \right)_{\bi} \\
        &+\sum_{j=1}^{s} a_{s,j}^{[I]} 
           \mathbf{L}_{\mathrm{net}}^{(j)}(\bv, \rho_{\mathrm{bc}})_{\bi}, 
        \end{split}
    \end{equation*}
     with $\adv^{(j)}$, $\lap^{(j)}$ 
     and $\mathbf{L}_{\mathrm{net}}^{(j)}$ the operators 
     corresponding to the embedded grid $\Omega^{(j)}$. 
    \Cref{eq:Helmholtz_type} is a Helmholtz-type linear system, 
     and we use the multigrid method in~\cite{Zhang2024Elliptic} 
     to solve it;
    \item Once the values for all six stages are obtained, 
     we update the solution at $t_{n+1}$ by 
    \begin{equation*}
        \avg{\rho}_{\bi}^{n+1} 
        =\frac{1}{\vol{{\cal C}_{\bi}^{(6)}}}\left(
            \vol{{\cal C}_{\bi}^{(6)}}\avg{\rho}_{\bi}^{(6)} + k \tilde{L}
         \right),
    \end{equation*}
     where 
     \begin{equation*}
        \tilde{L} 
        =-\sum_{j=1}^{6} \left(b_{j}-a_{6,j}^{[E]}\right) 
            \vol{{\cal C}_{\bi}^{(j)}} \adv^{(j)}\left(
                \avg{\rho}^{(j)}, \avg{\bu}^{(j)}
            \right)_{\bi}.
     \end{equation*}
\end{enumerate}

\section{Error analysis}
We analyze the one-step error of a cut cell to assess the accuracy 
 of the proposed method. 
Denote by $\overline{\Omega}^{(s)}$ the moving region at each stage 
 derived by the exact flow map with $\overline{\Omega}^{(0)} = \Omega^{(0)}$, 
 $\overline{{\cal C}}_{\bi}$ the corresponding cut cells, 
 and $\overline{\avg{\cdot}}_{\bi}$ the corresponding cell-averaged values. 
Assume the solution $\avg{\rho}_{\bi}^{n}$ is exact at time $t_n$, 
 we analyze the error between $\avg{\rho}_{\bi}^{n+1}$ 
 and $\overline{\avg{\rho(\bx,t_{n+1})}}_{\bi}$.
 
By triangular inequality, the error is expressed as 
\begin{equation}
    \label{eq:One_step_error}
    \begin{split}
    E^{n+1} \coloneqq& 
    \left|
    \overline{\avg{\rho(\bx,t_{n+1})}}_{\bi} - \avg{\rho}_{\bi}^{n+1}
    \right|\\ 
    \le& 
    \left|
        \overline{\avg{\rho(\bx,t_{n+1})}}_{\bi} - 
        \avg{\rho(\bx,t_{n+1})}_{\bi}
    \right| + 
    \left|\avg{\rho}_{\bi}^{n+1} - \avg{\rho(\bx,t_{n+1})}_{\bi}\right|.
    \end{split}
\end{equation}
For the first term on the right hand, 
 denote by $a$ the averaged value of $\rho(\bx, t_{n+1})$ 
 on ${\cal C}_{\bi}\cap\overline{\cal C}_{\bi}$, i.e., 
\begin{equation*}
    a \coloneqq 
    \frac{1}{\vol{{\cal C}_{\bi}\cap\overline{\cal C}_{\bi}}}
    \iint_{{\cal C}_{\bi}\cap\overline{\cal C}_{\bi}}\rho(\bx, t_{n+1})\ 
    \dif \bx. 
\end{equation*}
Since ${\cal C}_{\bi}\cup\overline{\cal C}_{\bi}$ 
 are contained in a $O(h)\times O(h)$ box, 
 we have $\rho(\bx,t_{n+1}) = a + O(h)$ 
 on ${\cal C}_{\bi}\cup\overline{\cal C}_{\bi}$ by Taylor expansion. 
Therefore, we have 
\begin{equation*}
    \begin{split}
    &\left|
        \overline{\avg{\rho(\bx,t_{n+1})}}_{\bi} - 
        \avg{\rho(\bx,t_{n+1})}_{\bi}
    \right| \\
    =& \left|
        \frac{1}{\vol{{\cal C}_{\bi}}}\iint_{{\cal C}_{\bi}} 
        \rho(\bx, t_{n+1})\ \dif\bx - 
        \frac{1}{\vol{\overline{\cal C}_{\bi}}}\iint_{\overline{\cal C}_{\bi}} 
        \rho(\bx, t_{n+1})\ \dif\bx
    \right|\\
    =& \left|
        \frac{1}{\vol{{\cal C}_{\bi}}} \left(
            \vol{{\cal C}_{\bi}}a + 
            \iint_{{\cal C}_{\bi}\backslash\overline{\cal C}_{\bi}} 
            O(h)\ \dif\bx
        \right) - 
        \frac{1}{\vol{\overline{\cal C}_{\bi}}}\left(
            \vol{\overline{\cal C}_{\bi}}a + 
            \iint_{\overline{\cal C}_{\bi}\backslash{\cal C}_{\bi}}
            O(h)\ \dif\bx
        \right)
    \right|\\
    =& \left|
        \frac{1}{\vol{{\cal C}_{\bi}}}
        \iint_{{\cal C}_{\bi}\backslash\overline{\cal C}_{\bi}} 
        O(h)\ \dif\bx - 
        \frac{1}{\vol{\overline{\cal C}_{\bi}}}
        \iint_{\overline{\cal C}_{\bi}\backslash{\cal C}_{\bi}}
        O(h)\ \dif\bx
    \right|.
    \end{split}
\end{equation*}
Since the one-step error of the ARMS method is $O(kh_L^4+k^5)$, 
 the volumes of both ${\cal C}_{\bi}\backslash\overline{\cal C}_{\bi}$ 
 and $\overline{\cal C}_{\bi}\backslash{\cal C}_{\bi}$ are $O(khh_L^4+k^5h)$. 
Therefore, 
\begin{equation*}
    \begin{split}
    \left|
        \overline{\avg{\rho(\bx,t_{n+1})}}_{\bi} - 
        \avg{\rho(\bx,t_{n+1})}_{\bi}
    \right|
    =&  \left|
        \frac{1}{\vol{{\cal C}_{\bi}}}O(kh^2h_L^4+k^5h^2) - 
        \frac{1}{\vol{\overline{\cal C}_{\bi}}}O(kh^2h_L^4+k^5h^2)
    \right| \\ 
    =& O(kh_L^4 + k^5),
    \end{split}
\end{equation*}
 where $\vol{{\cal C}_{\bi}}$ and $\vol{\overline{\cal C}_{\bi}}$ are $O(h^2)$. 

For the second term on the right hand side of \cref{eq:One_step_error}, 
 notice that 
 $\vol{{\cal C}_{\bi}}\left|
    \avg{\rho}_{\bi}^{n+1} - \avg{\rho(\bx,t_{n+1})}_{\bi}\right|$ 
 corresponds to the truncation error scaled by the time step size $k$. 
The time evolution variable in the ODE \cref{eq:AdvDiff_Finite_Volume} 
 can be represented as 
 $h^2\frac{\vol{{\cal C}_{\bi}(t)}}{h^2}\avg{\rho}_{\bi}$, 
 where $\frac{\vol{{\cal C}_{\bi}(t)}}{h^2}=O(1)$ is the volume of fraction. 
The IMEX scheme contributes a truncation error of $O(k^4h^2)$ 
 because of its fourth-order accuracy 
 and the $h^2$ factor in the time evolution variable. 
As for the spatial discretization, 
 the net transportation term $\mathbf{L}_\mathrm{net}$ 
 is computed by Simpson's rule, 
 achieving near machine precision, 
 so the error it introduces is negligible. 
Away from boundaries, 
 the spatial discretizations of $\adv$ and $\lap$ 
 are both fourth-order accurate, 
 cf. \cref{eq:Regular_adv} and \cref{eq:Regular_lap}. 
For cut cells near boundaries, 
 $\adv$ retains fourth-order accuracy, 
 but $\lap$ is only third-order accurate, 
 cf. \cref{eq:PLG4Adv} and \cref{eq:PLG4Lap}. 
Since both $\adv$ and $\lap$ are scaled by $\vol{{\cal C}_{\bi}}$ 
 in the ODE \cref{eq:AdvDiff_Finite_Volume}, 
 the spatial discretization introduces a truncation error of $O(h^6)$ 
 for most cut cells 
 and $O(h^5)$ for cut cells on a set of codimension $1$. 
Therefore, we have 
 $\left|\avg{\rho}_{\bi}^{n+1} - \avg{\rho(\bx,t_{n+1})}_{\bi}\right| 
  = O(k^5 + kh^4)$ in the $1$-norm 
 and the one-step error is $E^{n+1}=O(kh^4+kh_L^4+k^5)$ in the $1$-norm, 
 which confirming that the proposed method achieves fourth-order accuracy 
 in the $1$-norm.

\section{Numerical results}
\label{sec:Numerical_results}
In this section, 
 we test the proposed method with various numerical tests 
 to demonstrate its fourth-order accuracy 
 for the advection-diffusion equation with moving boundaries. 

The $L^p$ norm of a scalar function
 $g: \Omega\rightarrow \mathbb{R}$ is defined as
\begin{equation*}
  \|g\|_p = 
  \left\{
    \begin{aligned}
      &\left( 
        \sum\limits_{{\cal C}_{\bi}\subset \Omega} 
        \|{\cal C}_{\bi}\| \cdot 
        \left|\avg{g}_\bi\right|^p 
      \right)^{\frac{1}{p}} \quad 
      && \mathrm{if}\ p=1,2; \\
      & \max\limits_{{\cal C}_{\bi}\subset \Omega} |\avg{g}_\bi| \quad 
      && \mathrm{if}\ p=\infty, 
    \end{aligned}
  \right.
\end{equation*}
 where the cut cell ${\cal C}_{\bi}$ and the cell average $\avg{g}_\bi$ 
 are given in \Cref{sec:Algorithm}.  

\subsection{Problem 1: Diffusion from a point source}
We consider solving the advection-diffusion concentration field 
 initialized from a point source within a moving disk 
 from~\cite{Barrett2022}. 
The exact solution of the concentration field is given by
\begin{equation}
    \label{eq:PointSource}
    \rho(\bx,t)=\frac{10 \mathrm{Pe}}{4 \left(t + \frac{1}{2}\right)}
    \exp\left(
        -\mathrm{Pe}\frac{\|\bx-\bx_c(t)\|}{4\left(t+\frac{1}{2}\right)}
        \right),
\end{equation}
 where $\bx_c(t)$ represents the center of the moving disk. 

We consider two types of moving disk, 
 one is a translating disk, the other is a rotating disk. 
The computational domain, the moving region and the velocity field 
 are defined as follows:
\begin{enumerate}[({Test}-1)]
    \item $\Omega_{R}=[0,3]^2$, 
     $\Omega_{1}(t)=\{(x,y):(x-1-t)^2+(y-1-t)^2<1\}$, 
     $\bu=\bv=(
        \cos\left(\frac{\pi}{4}\right),\sin\left(\frac{\pi}{4}\right))^{\top}$;
    \item $\Omega_{R}=[-2.5,2.5]^2$, 
     $\Omega_{2}(t)=\{(x,y):(x-1.5\cos(2\pi t))^2+(y-1.5\sin(2\pi t))^2<1\}$, 
     $\bu=\bv=2\pi(-y,x)^{\top}$.
\end{enumerate} 
The Peclet number is $\mathrm{Pe}=10$ for both tests. 
The time step sizes $k$ are determined by the CFL condition 
 $C_\mathrm{CFL}\coloneqq\frac{\|\bu\|_2 k}{\epsilon h}\approx 0.17$ 
 and $0.5$, respectively,  
 where $\epsilon=0.15$ is the merge threshold in \Cref{sec:Merging_Alg}.

\Cref{tab:PointSourceDiffusion} shows the errors 
 in the $L^1,\,L^2$ and $L^\infty$ norms, 
 along with the corresponding convergence rates. 
These results confirm the fourth-order accuracy of our numerical method. 
\Cref{fig:PointSourceDiffusion} depicts the concentration field 
 at various time steps $t\in[0,1]$ for a grid resolution of $h=\frac{1}{80}$.

\begin{table}[htbp]
    \centering
    \begin{tabular}{c|c|c|c|c|c|c|c}
        \hline
        \multicolumn{8}{c}{Test-1}\\
        \hline
         & $h=\frac{1}{10}$ & rate & $h=\frac{1}{20}$ & rate 
         & $h=\frac{1}{40}$ & rate & $h=\frac{1}{80}$ \\
        \hline
        $L^\infty$ & 
        1.01e-02 & 4.13 & 5.76e-04 & 4.06 & 3.46e-05 & 4.03 & 2.13e-06 \\
        $L^1$ & 
        7.91e-03 & 4.04 & 4.82e-04 & 4.05 & 2.90e-05 & 4.03 & 1.78e-06 \\
        $L^2$ & 
        6.28e-03 & 4.03 & 3.85e-04 & 4.04 & 2.34e-05 & 4.03 & 1.44e-06 \\
        \hline
        \multicolumn{8}{c}{Test-2}\\
        \hline
         & $h=\frac{1}{10}$ & rate & $h=\frac{1}{20}$ & rate 
         & $h=\frac{1}{40}$ & rate & $h=\frac{1}{80}$ \\
        \hline
        $L^\infty$ & 
        9.17e-02 & 4.29 & 4.69e-03 & 4.05 & 2.83e-04 & 4.01 & 1.75e-05 \\
        $L^1$ & 
        1.16e-01 & 4.60 & 4.77e-03 & 4.06 & 2.86e-04 & 4.03 & 1.75e-05 \\
        $L^2$ & 
        8.15e-02 & 4.53 & 3.52e-03 & 4.05 & 2.13e-04 & 4.02 & 1.32e-05 \\
        \hline
    \end{tabular}
    \caption{Errors and convergence rates 
     for advection-diffusion concentration fields 
     initialized from a point source 
     within translating and rotating disks. 
    Simulations are run to a final time of $T=1$ 
     and the CFL numbers are $C_\mathrm{CFL}\approx 0.17$ and $0.5$, 
     respectively.
    }
    \label{tab:PointSourceDiffusion}
\end{table}

\begin{figure}[htb]
    \centering
    \begin{subfigure}{0.3\linewidth}
		\centering
		\includegraphics[width=0.9\linewidth]{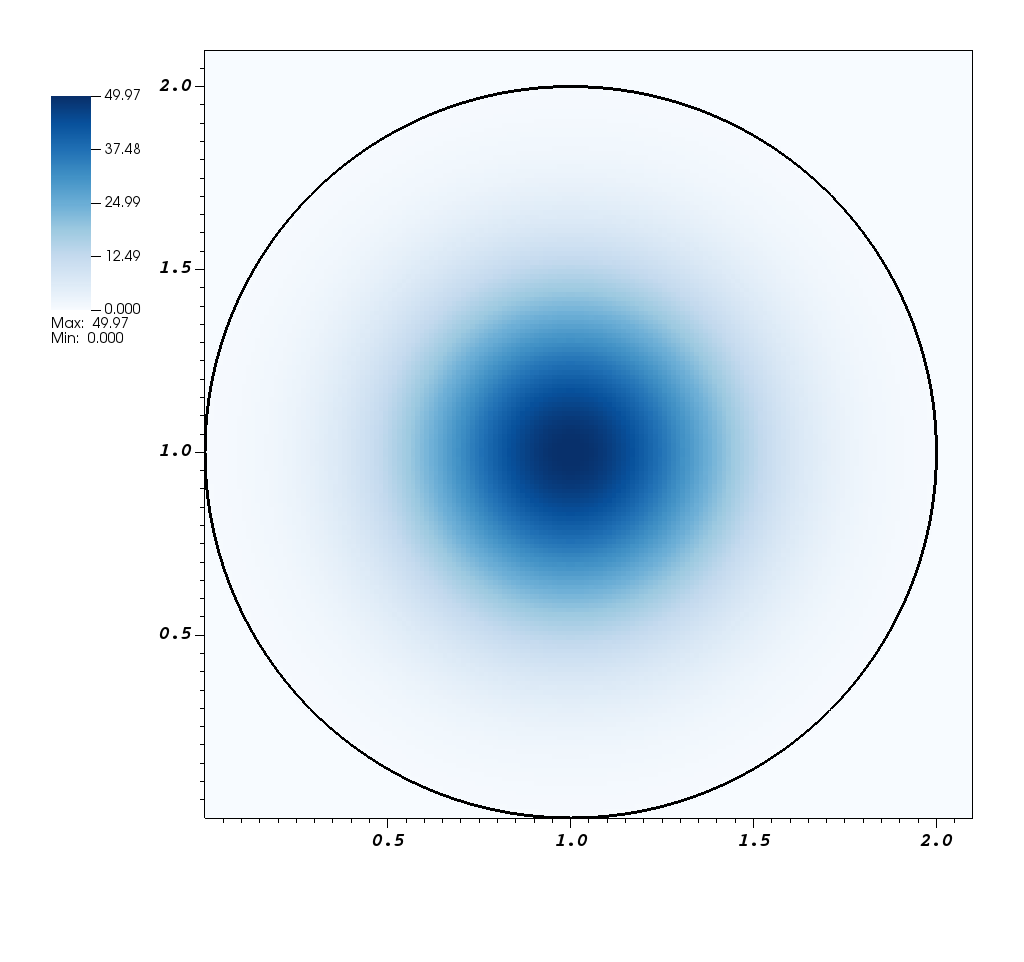}
		\caption{Test-1: $t=0$}
	\end{subfigure}
	\begin{subfigure}{0.3\linewidth}
		\centering
		\includegraphics[width=0.9\linewidth]{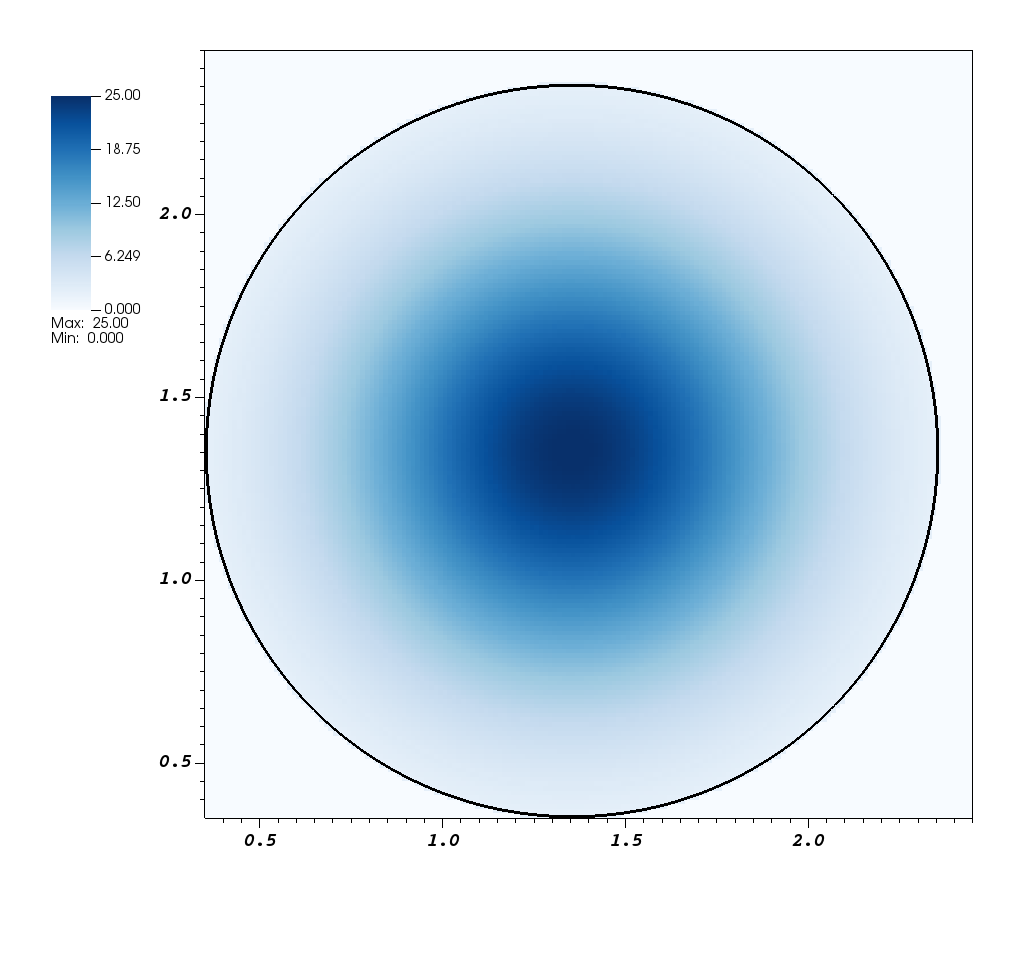}
		\caption{Test-1: $t=0.5$}
	\end{subfigure}
	\begin{subfigure}{0.3\linewidth}
		\centering
		\includegraphics[width=0.9\linewidth]{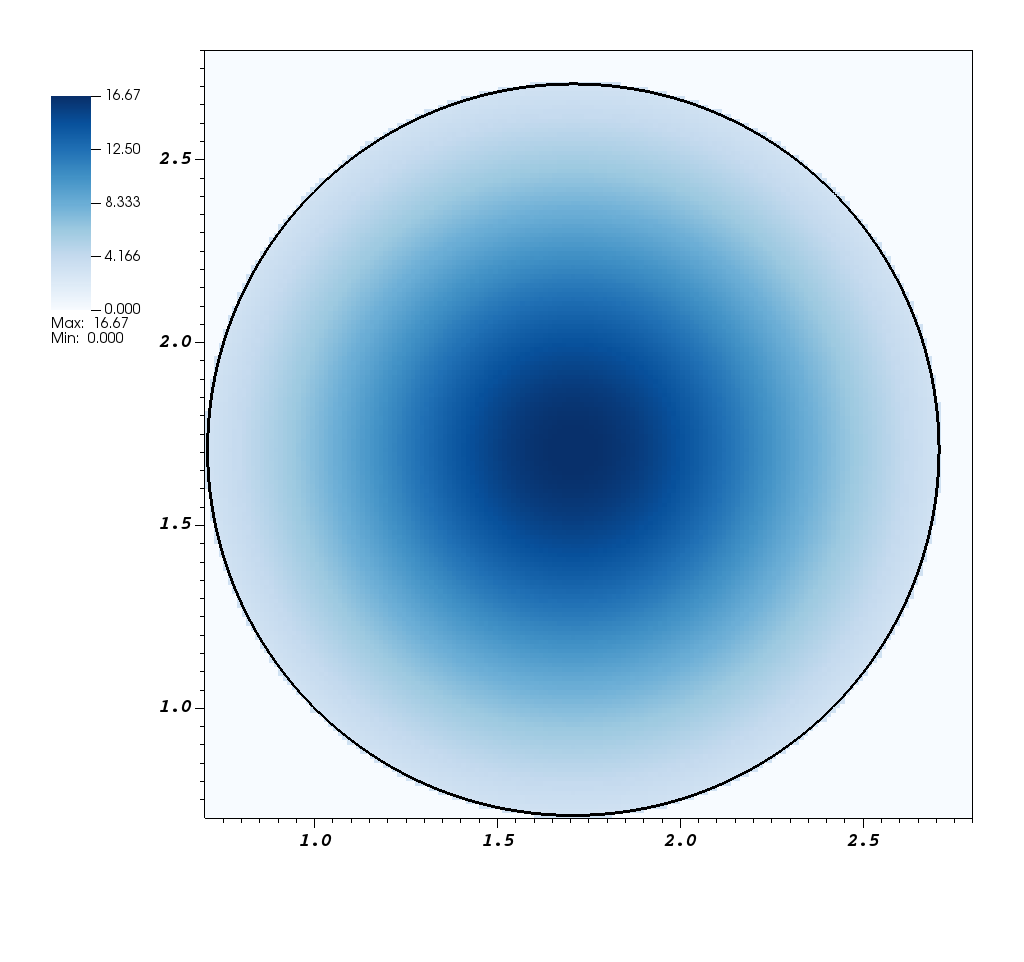}
		\caption{Test-1: $t=1$}
	\end{subfigure}\\
    \begin{subfigure}{0.3\linewidth}
		\centering
		\includegraphics[width=0.9\linewidth]{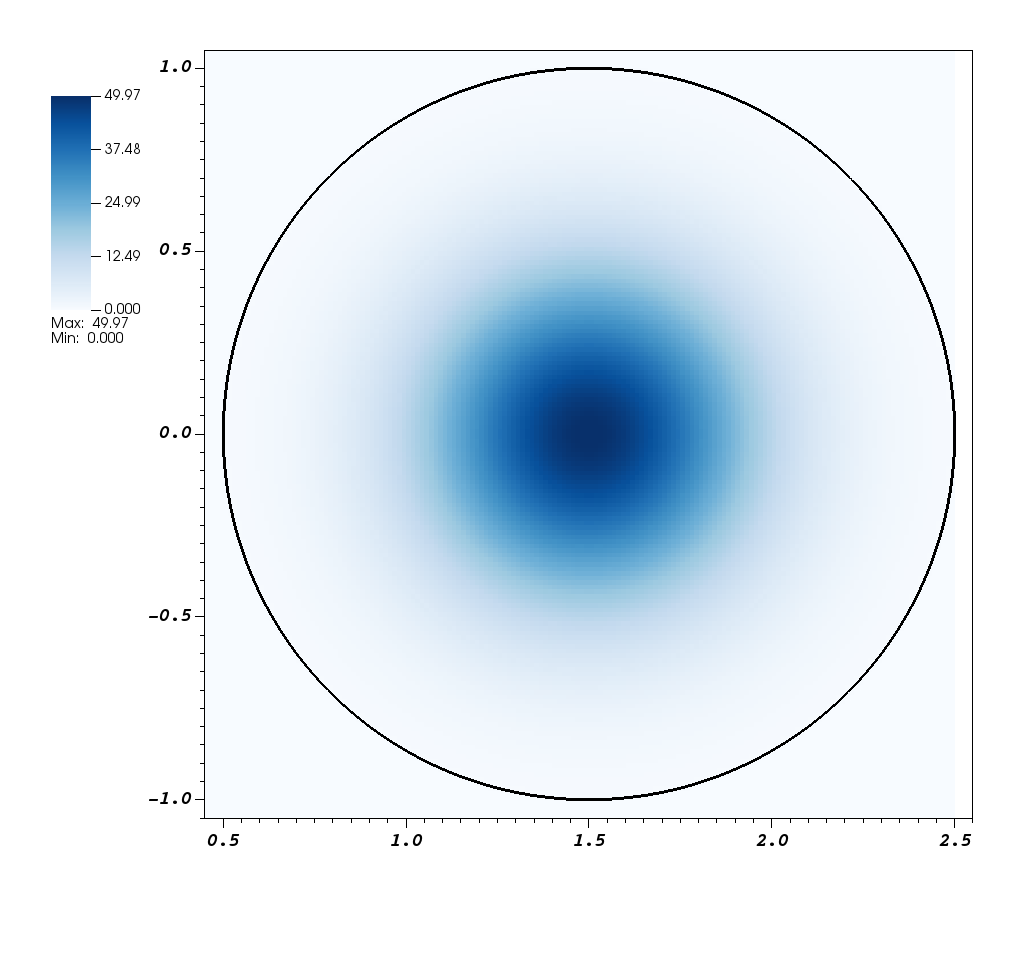}
		\caption{Test-2: $t=0$}
	\end{subfigure}
	\begin{subfigure}{0.3\linewidth}
		\centering
		\includegraphics[width=0.9\linewidth]{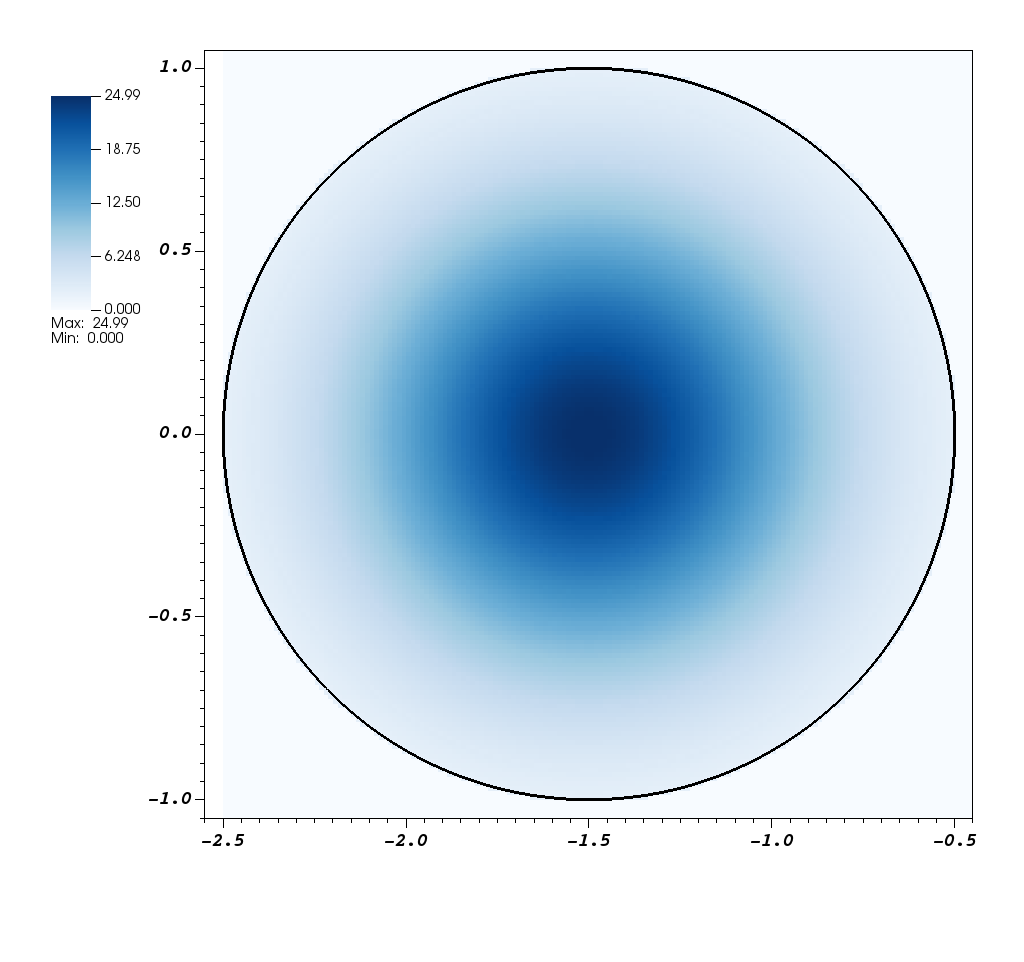}
		\caption{Test-2: $t=0.5$}
	\end{subfigure}
	\begin{subfigure}{0.3\linewidth}
		\centering
		\includegraphics[width=0.9\linewidth]{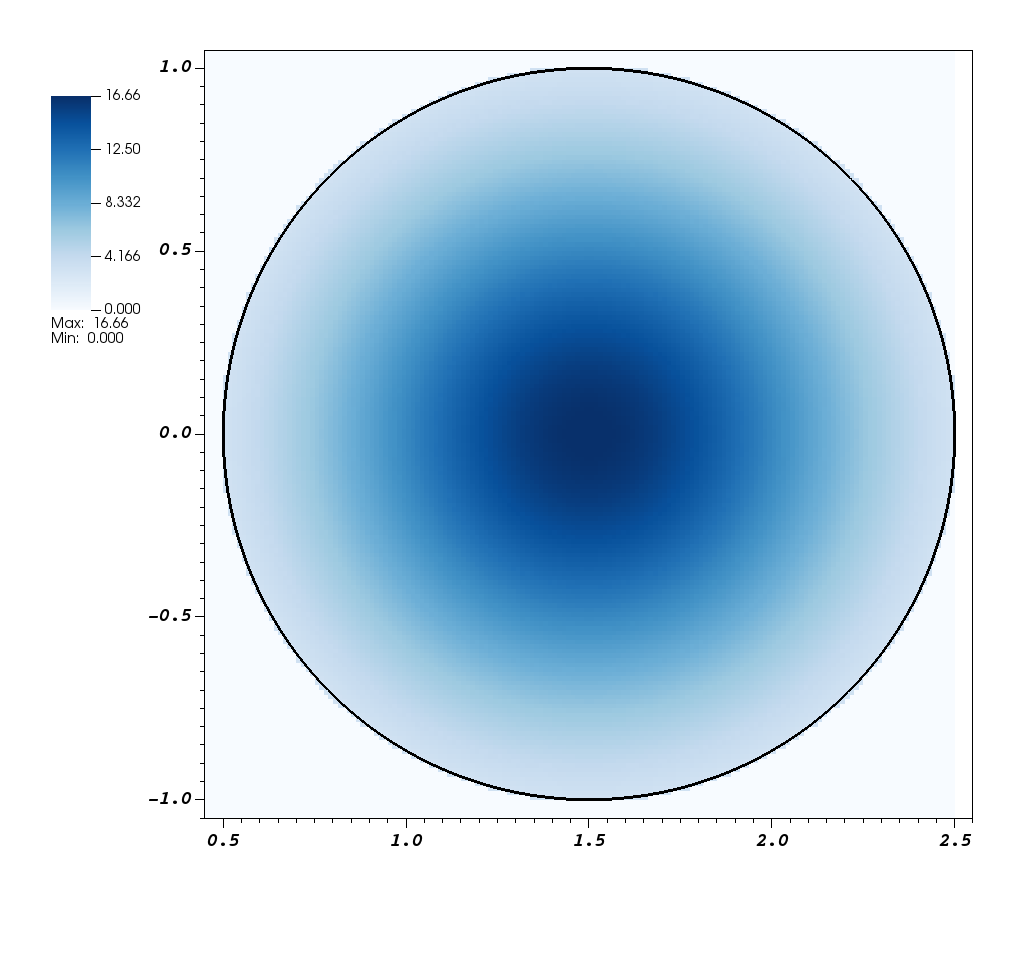}
		\caption{Test-2: $t=1$}
	\end{subfigure}
    \caption{Numerical concentration fields at different times. 
    The moving boundary is depicted in a black curve. 
    The simulation is run to a final time of $T=1$ 
     with a grid resolution of $h=1/80$, 
     the CFL numbers are $C_\mathrm{CFL}\approx0.17$ and $0.5$ , respectively.}
    \label{fig:PointSourceDiffusion}
\end{figure}

\subsection{Problem 2: Translation and scaling of a disk}
We now consider solving the advection-diffusion equation 
 where $\bu\neq\bv$. 
The exact solution of the concentration field 
 and the advection velocity are 
\begin{equation}
    \label{eq:DecaySinSin}
    \rho(x,y,t)=e^{-\frac{2t}{\mathrm{Pe}}}\sin(x-t)\sin(y-t); \quad 
    \bu=(1,1)^{\top}.
\end{equation}

We consider three types of moving region, 
 a translating disk, a contracting disk and an expanding disk. 
The computational domain, the moving region and the boundary velocity 
 for the tests are defined as follows 
\begin{enumerate}[({Test}-1)]
    \item \label{enum:Translation}
     $\Omega_{R}=[-1,2]^2$, 
     $\Omega_{1}(t)=\{(x,y):(x-0.5t)^2+(y-0.5t)^2<1\}$, 
     $\bv=(0.5,0.5)^{\top}$;
    \item \label{enum:Contraction}
     $\Omega_{R}=[-1.5,1.5]^2$, 
     $\Omega_{2}(t) = \{(x,y):x^2+y^2<(1.5-0.5t)^2\}$, 
     $\bv=-\frac{0.5}{\sqrt{x^2+y^2}}(x,y)^{\top}$;
    \item \label{enum:Expansion}
     $\Omega_{R}=[-1.5,1.5]^2$, 
     $\Omega_{3}(t)=\{(x,y):x^2+y^2<(1 + 0.5t)^2\}$, 
     $\bv=\frac{0.5}{\sqrt{x^2+y^2}}(x,y)^{\top}$.
\end{enumerate}
The Peclet number is $\mathrm{Pe}=10$ for all tests. 
The time step size $k$ is chosen by the CFL condition 
 $C_{\mathrm{CFL}}\approx 0.05$, 
 where the merge threshold is $\epsilon = 0.15$.

\Cref{tab:DecaySinSin} shows the errors 
 in the $L^1,\,L^2$ and $L^\infty$ norms, 
 along with the corresponding convergence rates. 
These results demonstrate the fourth-order accuracy of our numerical method. 
\Cref{fig:DecaySinSin} depicts the concentration field 
 at various time steps $t\in[0,1]$ for a grid resolution of $h=\frac{1}{80}$.

\begin{table}[htbp]
    \centering
    \begin{tabular}{c|c|c|c|c|c|c|c}
        \hline
        \multicolumn{8}{c}{Test-\ref{enum:Translation}}\\
        \hline
         & $h=\frac{1}{10}$ & rate & $h=\frac{1}{20}$ & rate 
         & $h=\frac{1}{40}$ & rate & $h=\frac{1}{80}$ \\
        \hline
        $L^\infty$ & 
        1.56e-06 & 3.85 & 1.08e-07 & 3.92 & 7.12e-09 & 3.98 & 4.53e-10 \\
        $L^1$ & 
        2.09e-06 & 3.70 & 1.61e-07 & 3.91 & 1.07e-08 & 4.00 & 6.70e-10 \\
        $L^2$ & 
        1.44e-06 & 3.73 & 1.09e-07 & 3.92 & 7.19e-09 & 3.99 & 4.51e-10 \\
        \hline\hline
        \multicolumn{8}{c}{Test-\ref{enum:Contraction}}\\
        \hline
         & $h=\frac{1}{10}$ & rate & $h=\frac{1}{20}$ & rate 
         & $h=\frac{1}{40}$ & rate & $h=\frac{1}{80}$ \\
        \hline
        $L^\infty$ & 
        2.37e-06 & 3.96 & 1.52e-07 & 3.98 & 9.67e-09 & 3.99 & 6.09e-10 \\
        $L^1$ & 
        3.60e-06 & 3.87 & 2.47e-07 & 3.95 & 1.60e-08 & 3.99 & 1.01e-09 \\
        $L^2$ & 
        2.44e-06 & 3.88 & 1.66e-07 & 3.95 & 1.07e-08 & 3.99 & 6.72e-10 \\
        \hline\hline
        \multicolumn{8}{c}{Test-\ref{enum:Expansion}}\\
        \hline
         & $h=\frac{1}{10}$ & rate & $h=\frac{1}{20}$ & rate 
         & $h=\frac{1}{40}$ & rate & $h=\frac{1}{80}$ \\
        \hline
        $L^\infty$ & 
        2.33e-06 & 4.15 & 1.32e-07 & 3.95 & 8.52e-09 & 3.98 & 5.41e-10 \\
        $L^1$ & 
        6.21e-06 & 3.94 & 4.04e-07 & 3.98 & 2.55e-08 & 4.02 & 1.57e-09 \\
        $L^2$ & 
        2.91e-06 & 3.93 & 1.91e-07 & 3.98 & 1.21e-08 & 4.01 & 7.51e-10 \\
        \hline
    \end{tabular}
    \caption{Errors and convergence rates for 
     Test-\ref{enum:Translation}\textemdash Test-\ref{enum:Expansion}. 
    Simulations are run to a final time of $T=1$ 
     and the CFL number is $C_\mathrm{CFL}\approx 0.05$.
    }
    \label{tab:DecaySinSin}
\end{table}

\begin{figure}[htb]
    \centering
    \begin{subfigure}{0.3\linewidth}
		\centering
		\includegraphics[width=0.9\linewidth]{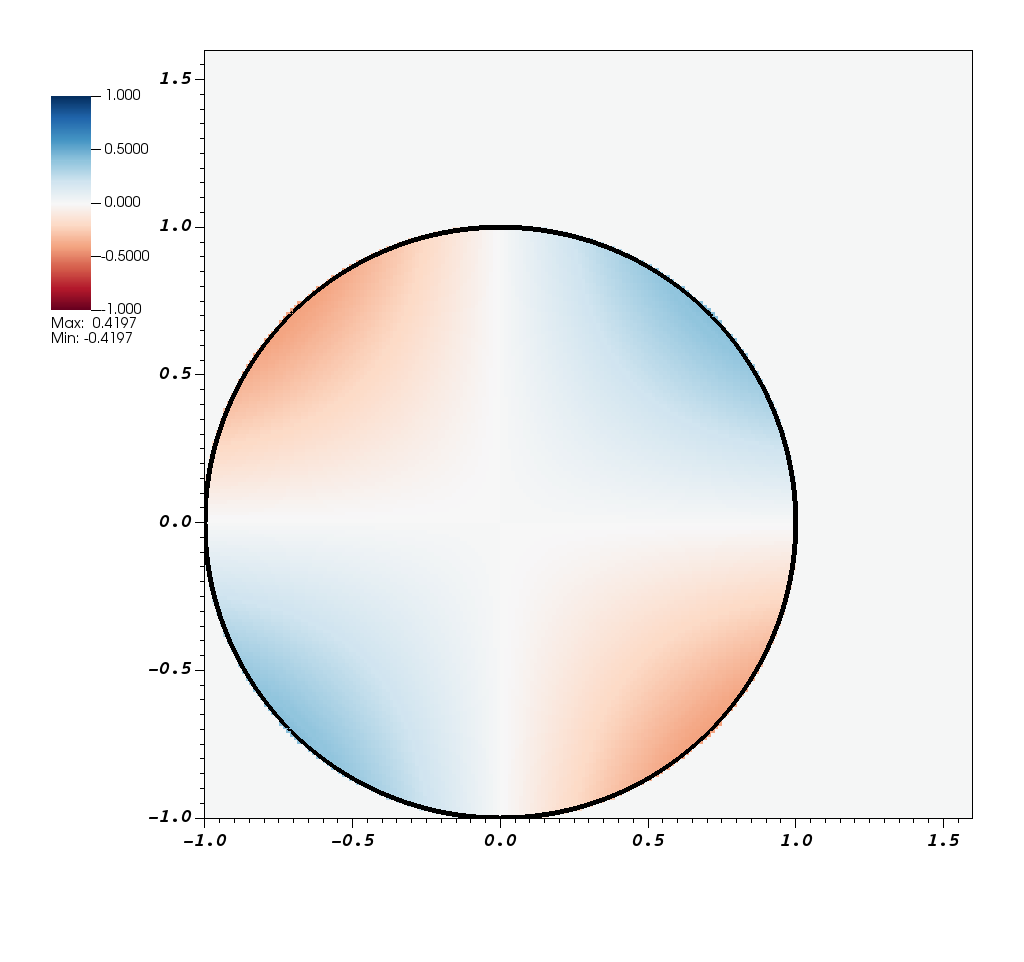}
		\caption{Test-\ref{enum:Translation}: $t=0$}
	\end{subfigure}
	\begin{subfigure}{0.3\linewidth}
		\centering
		\includegraphics[width=0.9\linewidth]{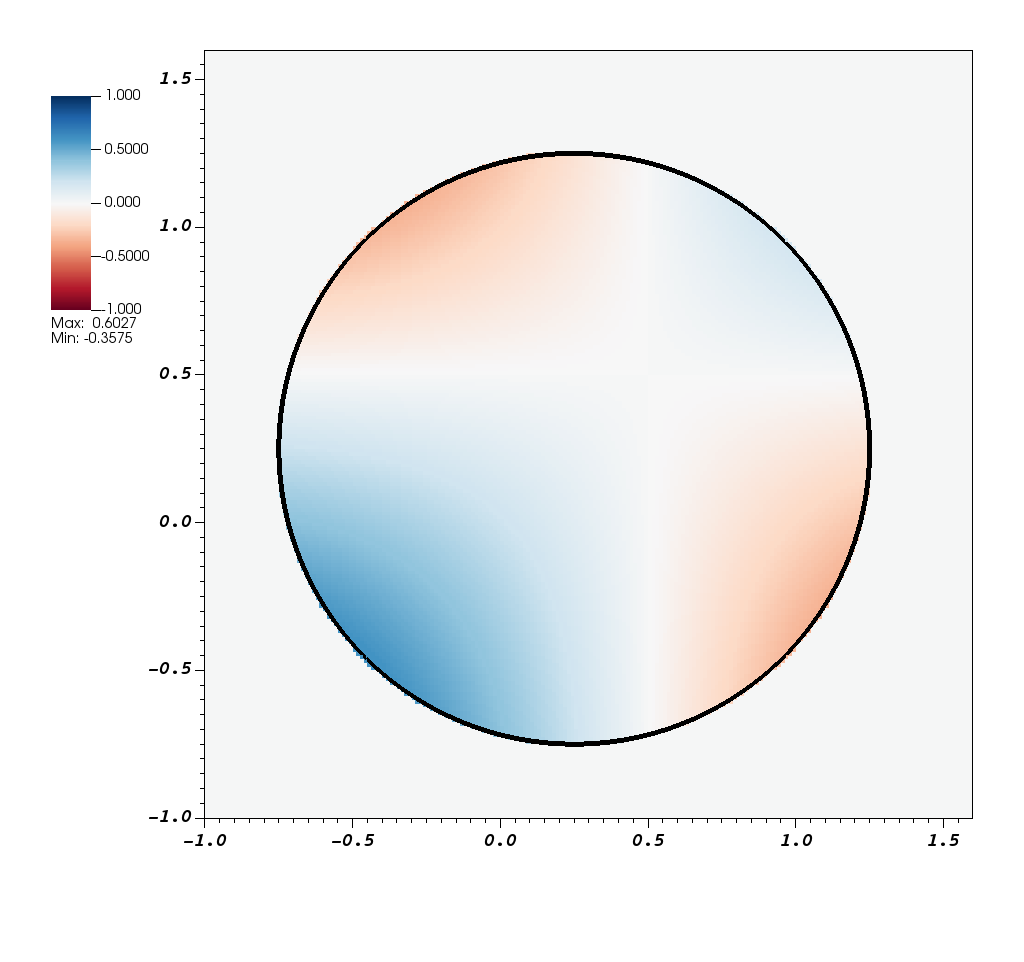}
		\caption{Test-\ref{enum:Translation}: $t=0.5$}
	\end{subfigure}
	\begin{subfigure}{0.3\linewidth}
		\centering
		\includegraphics[width=0.9\linewidth]{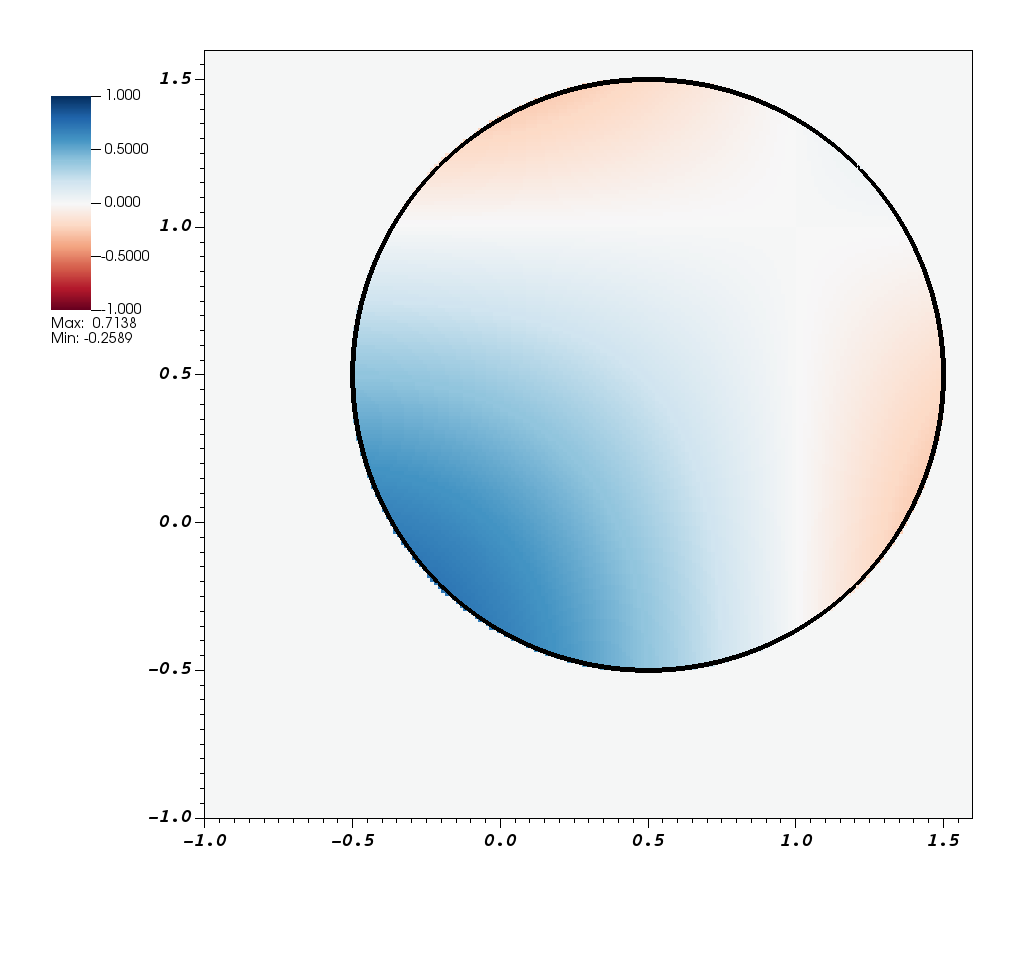}
		\caption{Test-\ref{enum:Translation}: $t=1$}
	\end{subfigure}\\
    \begin{subfigure}{0.3\linewidth}
		\centering
		\includegraphics[width=0.9\linewidth]{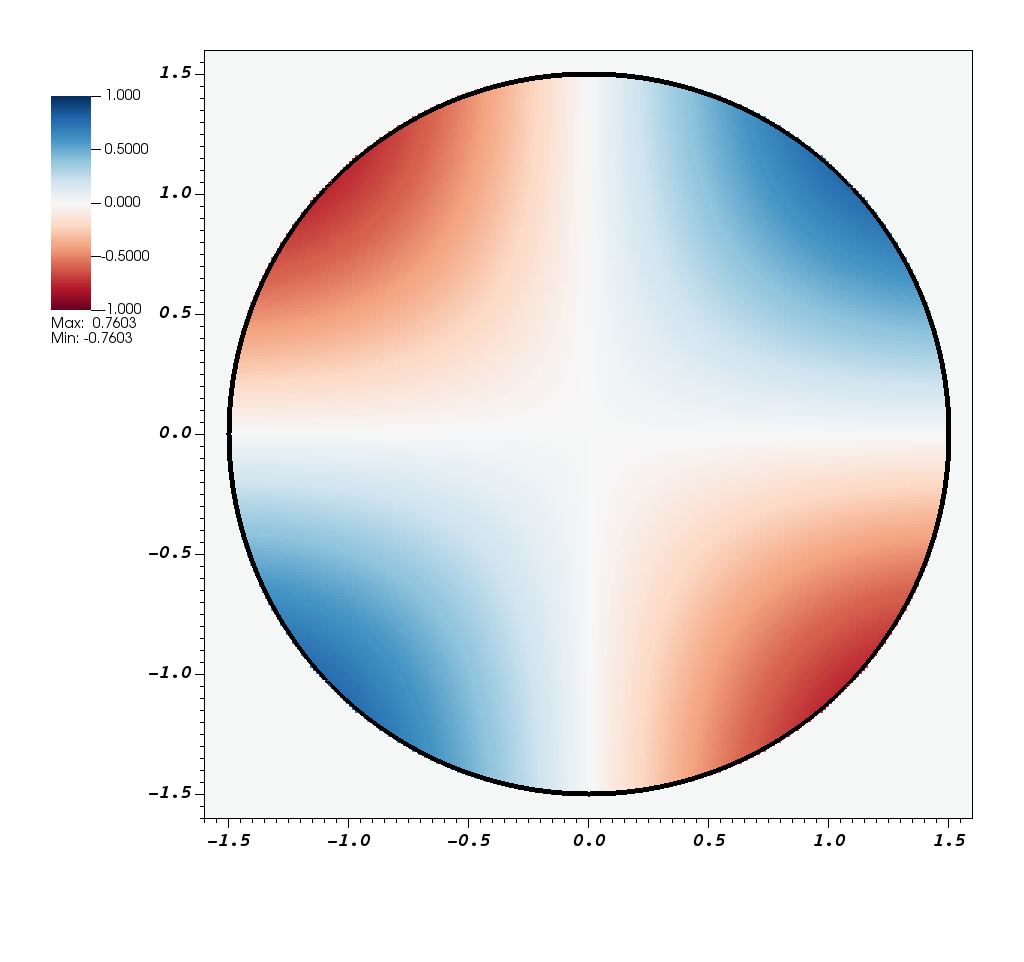}
		\caption{Test-\ref{enum:Contraction}: $t=0$}
	\end{subfigure}
	\begin{subfigure}{0.3\linewidth}
		\centering
		\includegraphics[width=0.9\linewidth]{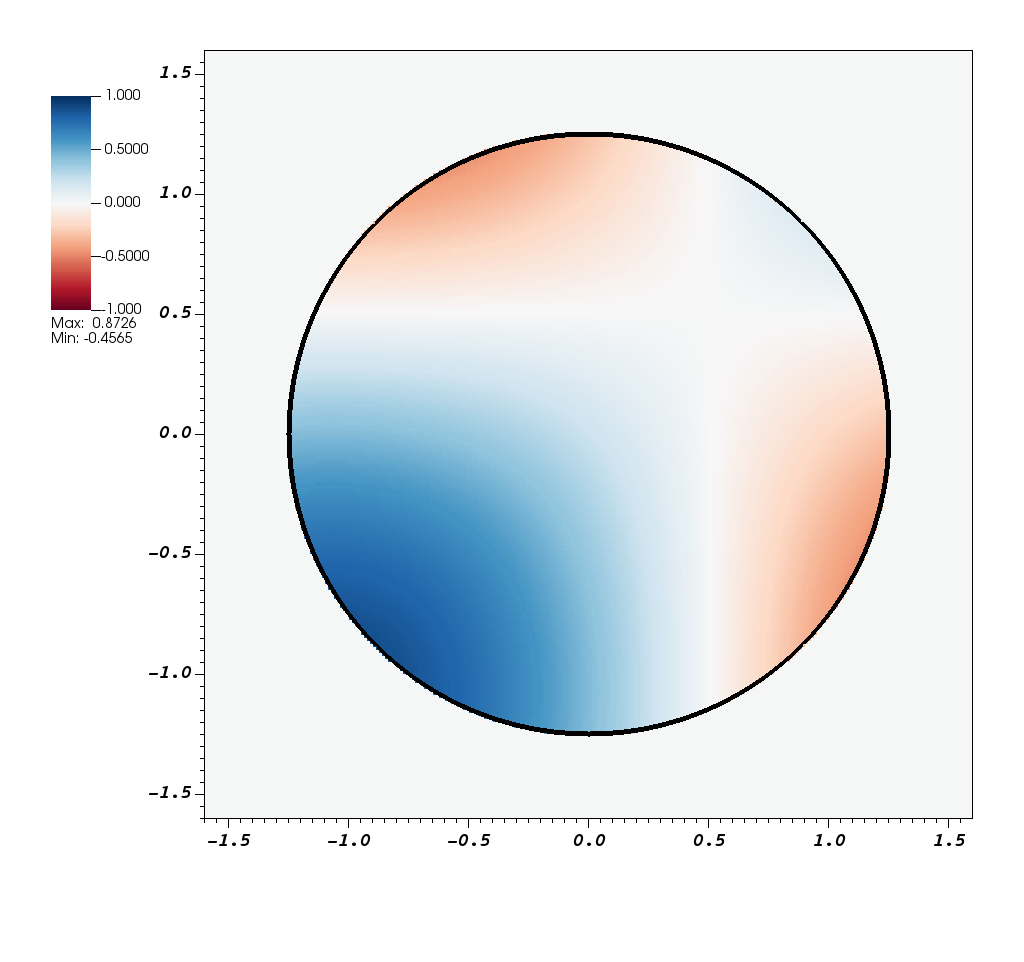}
		\caption{Test-\ref{enum:Contraction}: $t=0.5$}
	\end{subfigure}
	\begin{subfigure}{0.3\linewidth}
		\centering
		\includegraphics[width=0.9\linewidth]{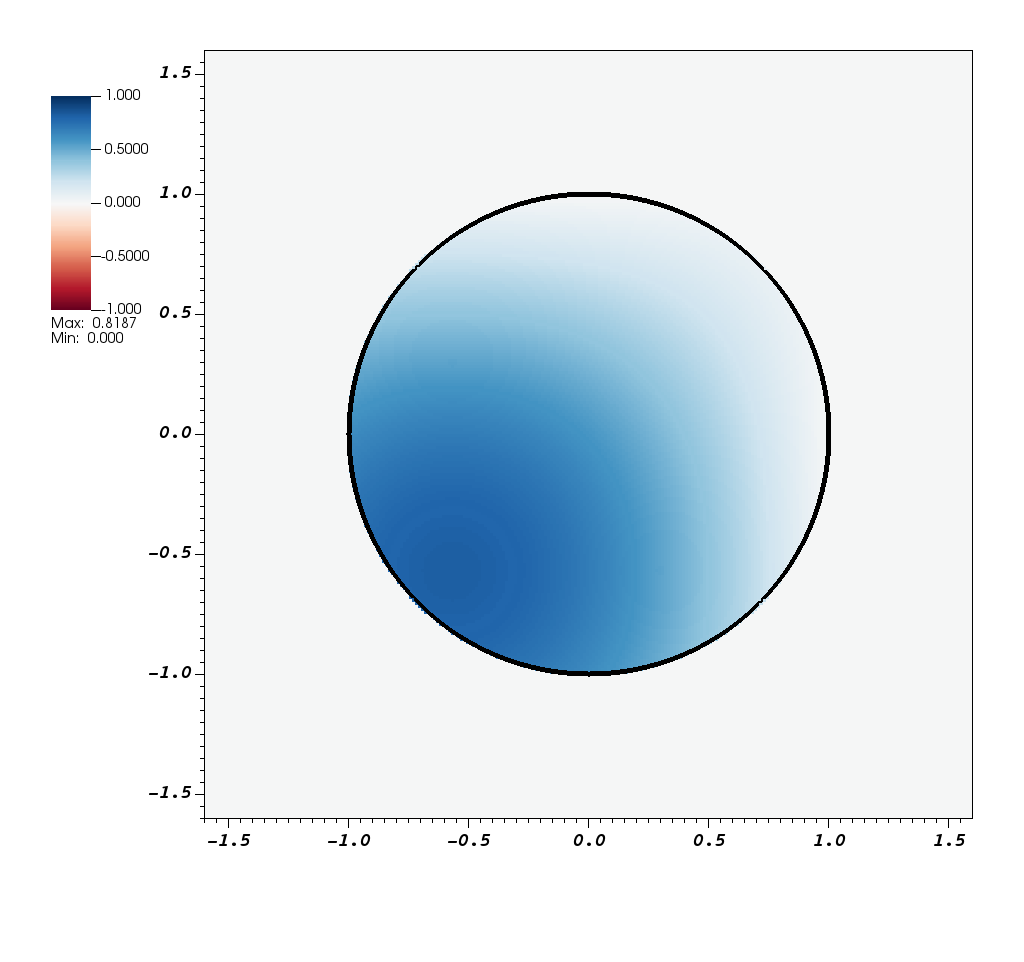}
		\caption{Test-\ref{enum:Contraction}: $t=1$}
	\end{subfigure}\\
    \begin{subfigure}{0.3\linewidth}
		\centering
		\includegraphics[width=0.9\linewidth]{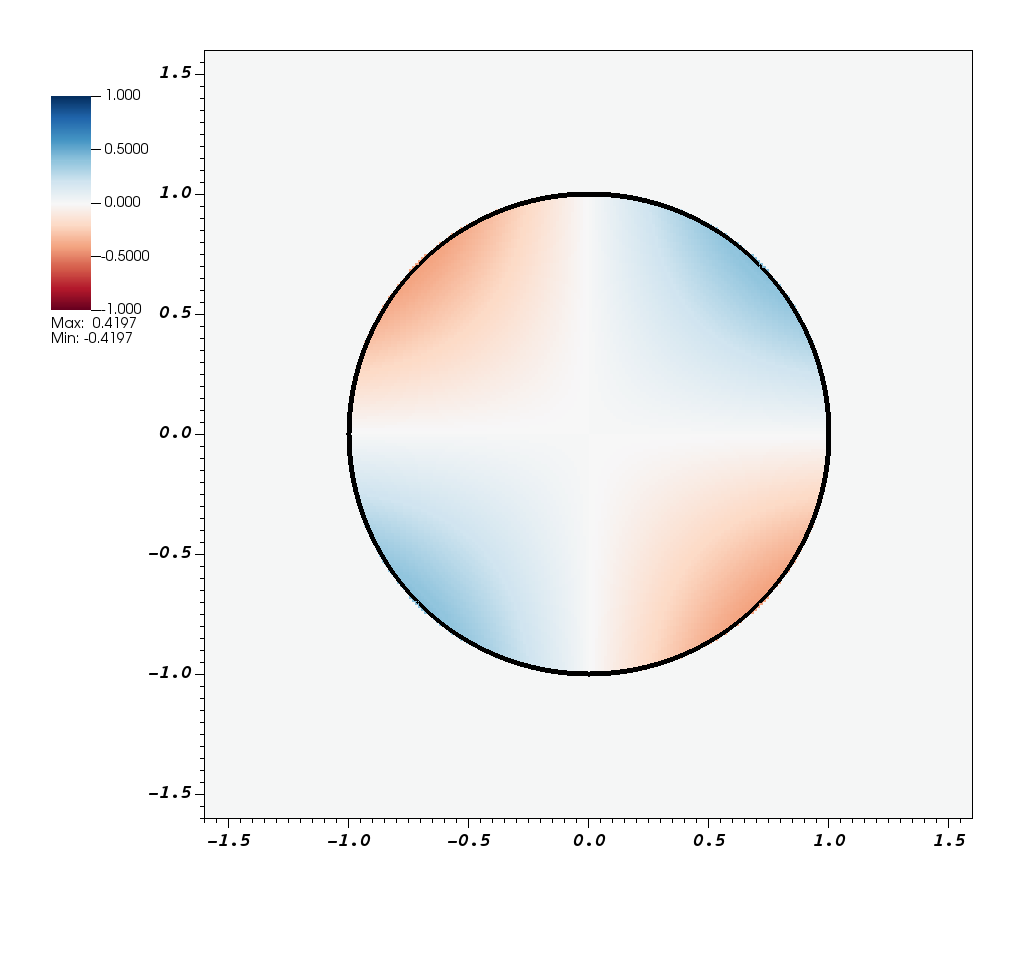}
		\caption{Test-\ref{enum:Expansion}: $t=0$}
	\end{subfigure}
	\begin{subfigure}{0.3\linewidth}
		\centering
		\includegraphics[width=0.9\linewidth]{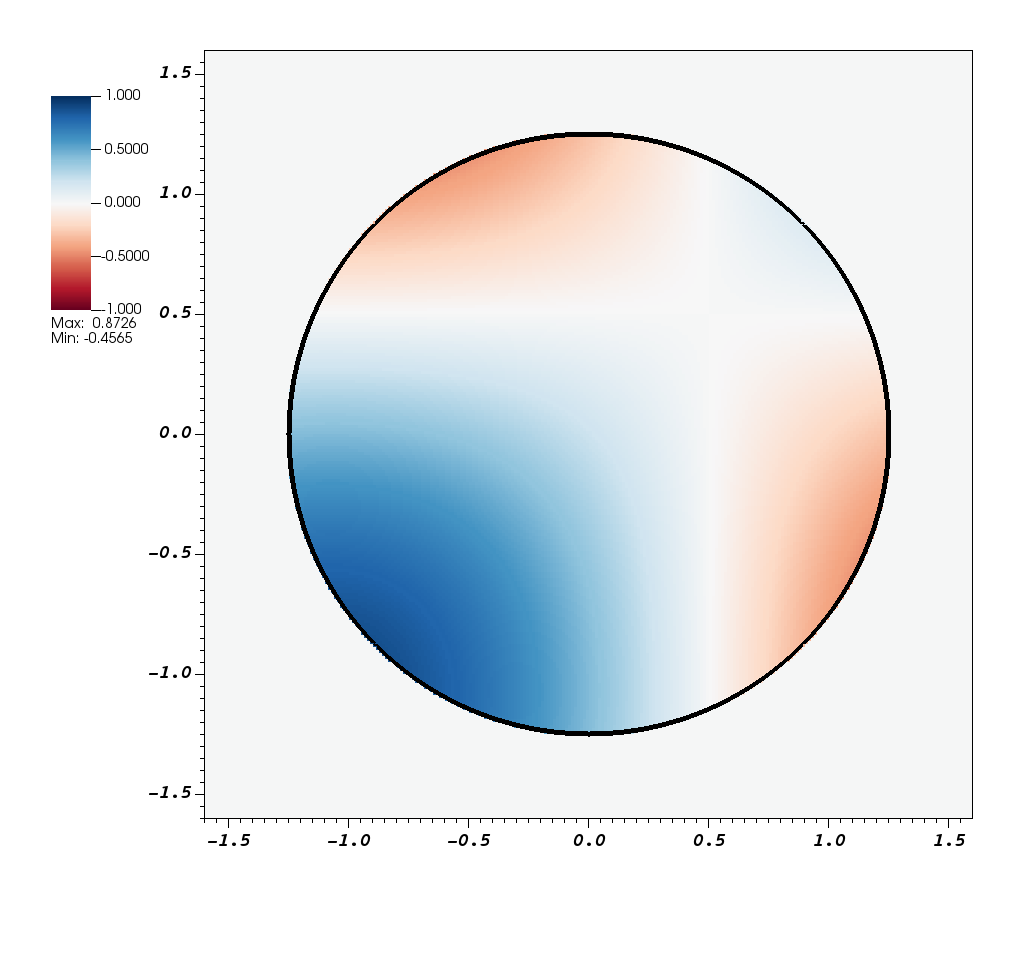}
		\caption{Test-\ref{enum:Expansion}: $t=0.5$}
	\end{subfigure}
	\begin{subfigure}{0.3\linewidth}
		\centering
		\includegraphics[width=0.9\linewidth]{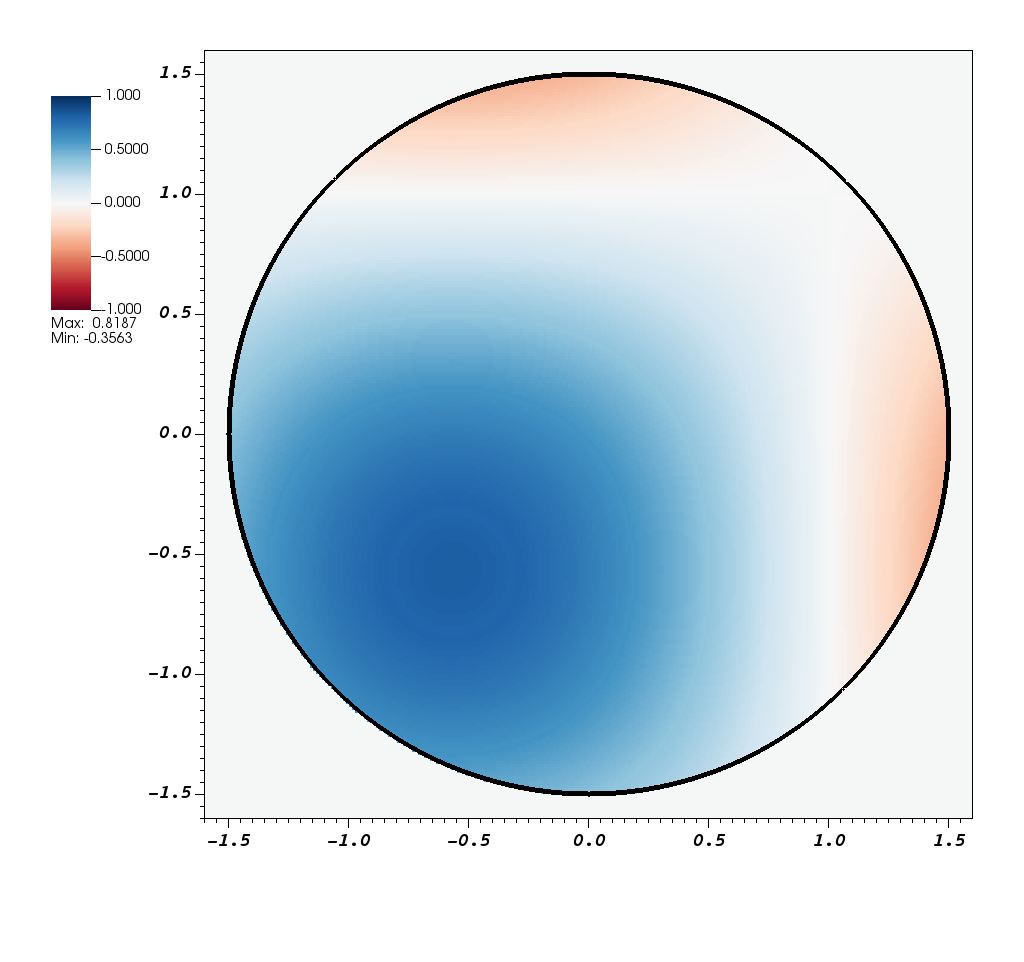}
		\caption{Test-\ref{enum:Expansion}: $t=1$}
	\end{subfigure}
    \caption{Numerical concentration fields at different times. 
    The moving boundary is depicted in a black curve. 
    The simulation is run to a final time of $T=1$ 
     with a grid resolution of $h=1/80$ 
     and the CFL number of $C_\mathrm{CFL}\approx 0.05$.}
    \label{fig:DecaySinSin}
\end{figure}

\subsection{Problem 3: Vortex flow}
We now examine our method for solving the advection-diffusion equation 
 on a deformed moving region. 
The initial domain $\Omega(0)$ is a disk 
 with a radius of $0.9$ and centered at $(2,3)$. 
The disk is driven by a vortex flow given by 
\begin{equation}
    \label{eq:VelOfVortexFlow}
    \bv=\cos\left(\frac{\pi t}{8}\right)\left(
        \sin^2\left(\frac{\pi}{4} x\right)\sin\left(\frac{\pi}{2} y\right), 
        -\sin^2\left(\frac{\pi}{4} y\right)\sin\left(\frac{\pi}{2} x\right)
     \right)^{\top}.
\end{equation}
The exact concentration field and the advection velocity 
 are the same as \cref{eq:DecaySinSin}. 
The computation domain is $\Omega_{R}=(0,3)\times(1,4)$, 
 the Peclet number is $\mathrm{Pe}=10$, 
 and the time step size $k$ is chosen by the CFL condition 
 $C_{\mathrm{CFL}}\coloneqq \frac{\max\|\bu\|_2 k}{\epsilon h}\approx 0.05$, 
 the merge threshold is $\epsilon=0.15$. 
The simulation is run to a final time of $T=2$.

\Cref{tab:VortexFlow} shows the errors 
 in the $L^1,\,L^2$ and $L^\infty$ norms, 
 along with the corresponding convergence rates. 
The result confirms the fourth-order accuracy of our numerical method 
 when handling deformed moving regions. 
\Cref{fig:VortexFlow} depicts the concentration field 
 at various time steps $t\in[0,2]$ for a grid resolution of $h=\frac{1}{80}$.

\begin{table}[htbp]
    \centering
    \begin{tabular}{c|c|c|c|c|c|c|c}
        \hline
         & $h=\frac{1}{10}$ & rate & $h=\frac{1}{20}$ & rate 
         & $h=\frac{1}{40}$ & rate & $h=\frac{1}{80}$ \\
        \hline
        $L^\infty$ & 
        9.38e-07 & 3.75 & 6.98e-08 & 4.49 & 3.10e-09 & 3.80 & 2.22e-10 \\
        $L^1$ & 
        4.42e-07 & 4.40 & 2.10e-08 & 3.49 & 1.87e-09 & 3.75 & 1.39e-10 \\
        $L^2$ & 
        3.88e-07 & 4.51 & 1.71e-08 & 3.49 & 1.52e-09 & 3.76 & 1.12e-10 \\
        \hline
    \end{tabular}
    \caption{Errors and convergence rates 
     for solving advection-diffusion equation 
     on a deformed moving region. 
    Simulations are run to a final time of $T=2$ 
     and the CFL number is $C_\mathrm{CFL}\approx 0.05$.
    }
    \label{tab:VortexFlow}
\end{table}

\begin{figure}[htb]
    \centering
    \begin{subfigure}{0.3\linewidth}
		\centering
		\includegraphics[width=0.9\linewidth]{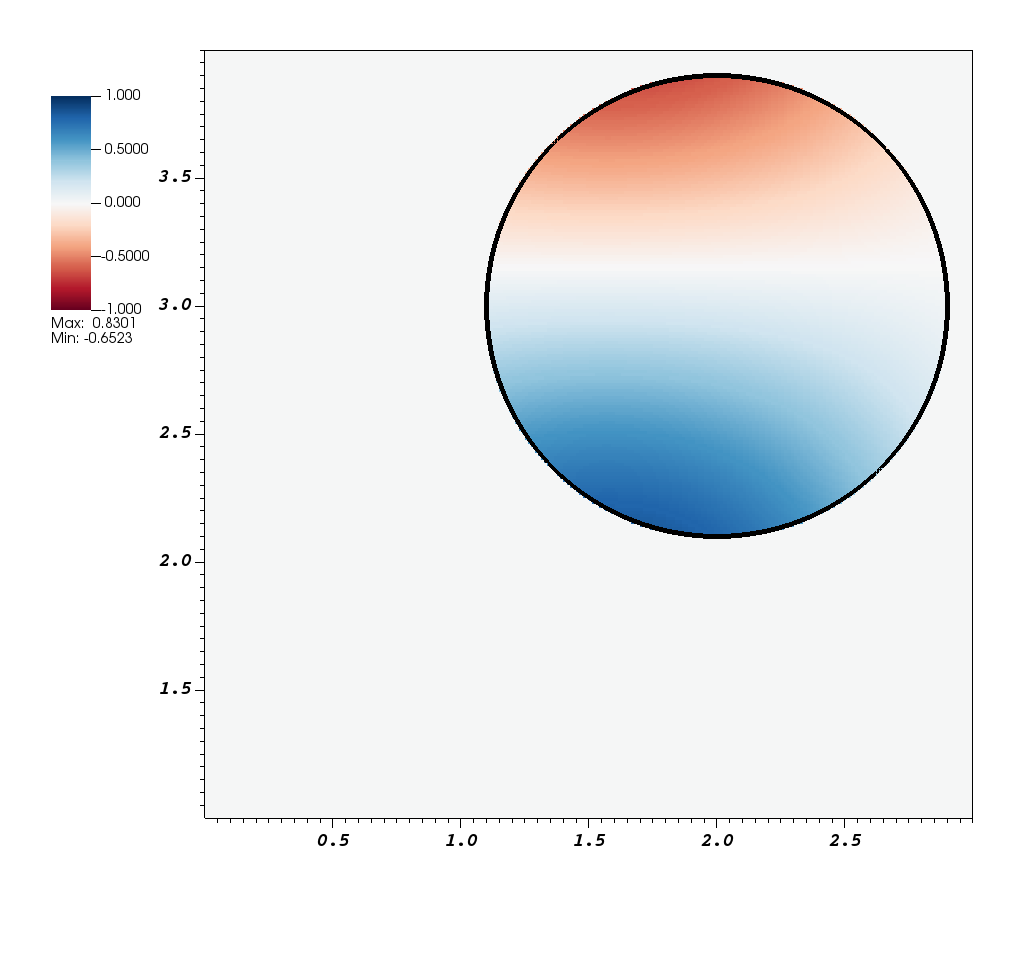}
		\caption{$t=0$}
	\end{subfigure}
	\begin{subfigure}{0.3\linewidth}
		\centering
		\includegraphics[width=0.9\linewidth]{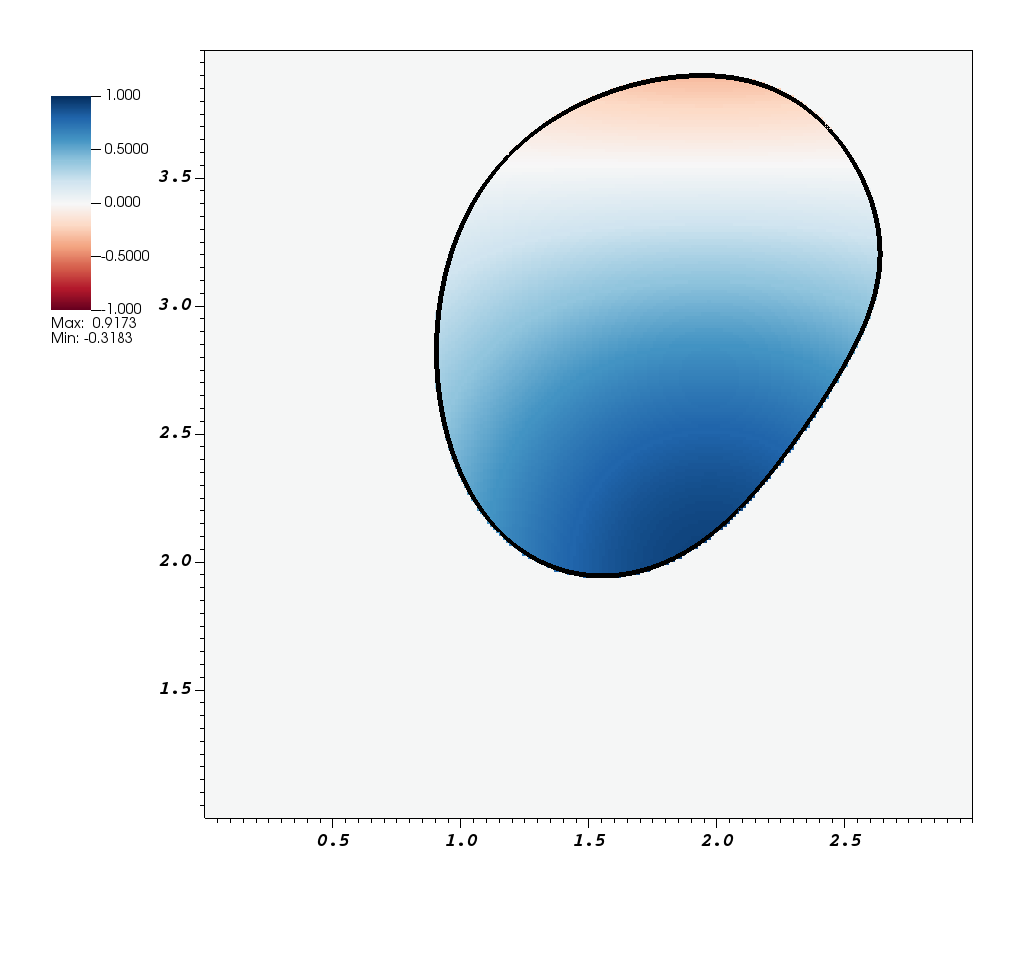}
		\caption{Test-1: $t=0.4$}
	\end{subfigure}
	\begin{subfigure}{0.3\linewidth}
		\centering
		\includegraphics[width=0.9\linewidth]{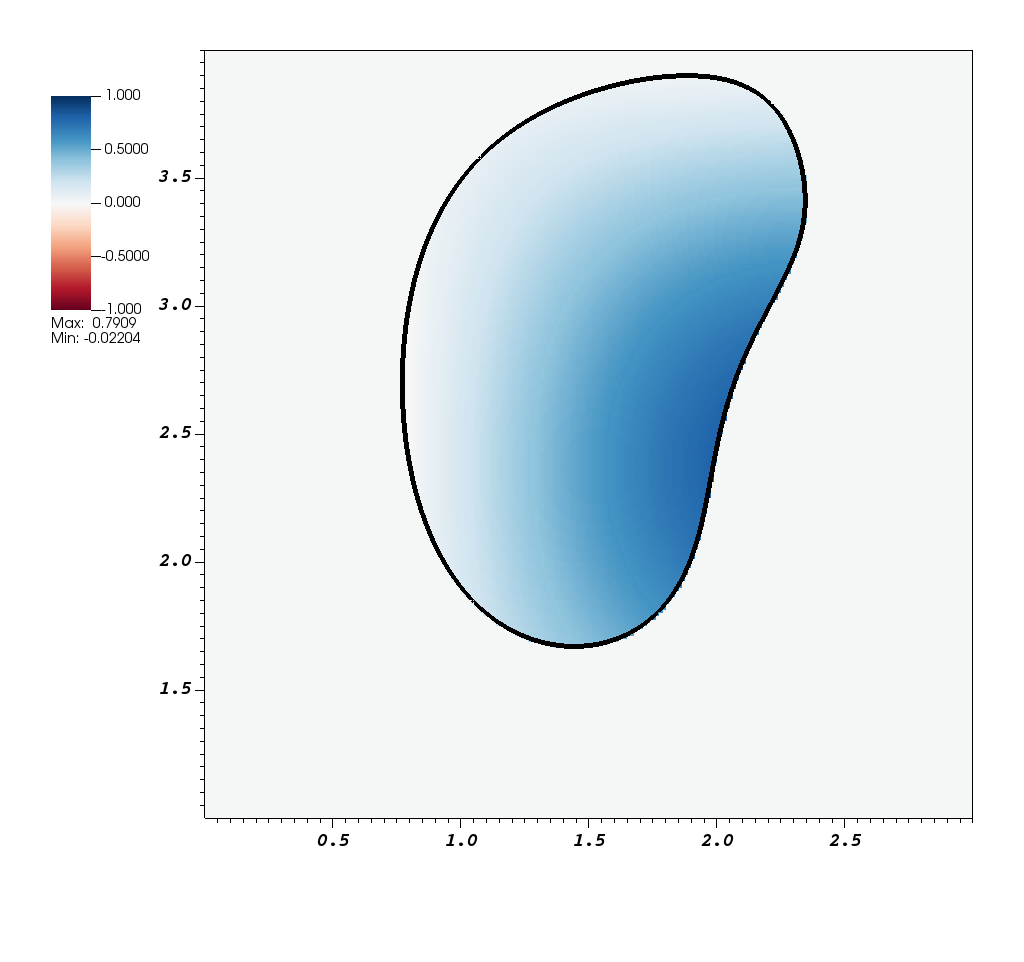}
		\caption{Test-1: $t=0.8$}
	\end{subfigure}\\
    \begin{subfigure}{0.3\linewidth}
		\centering
		\includegraphics[width=0.9\linewidth]{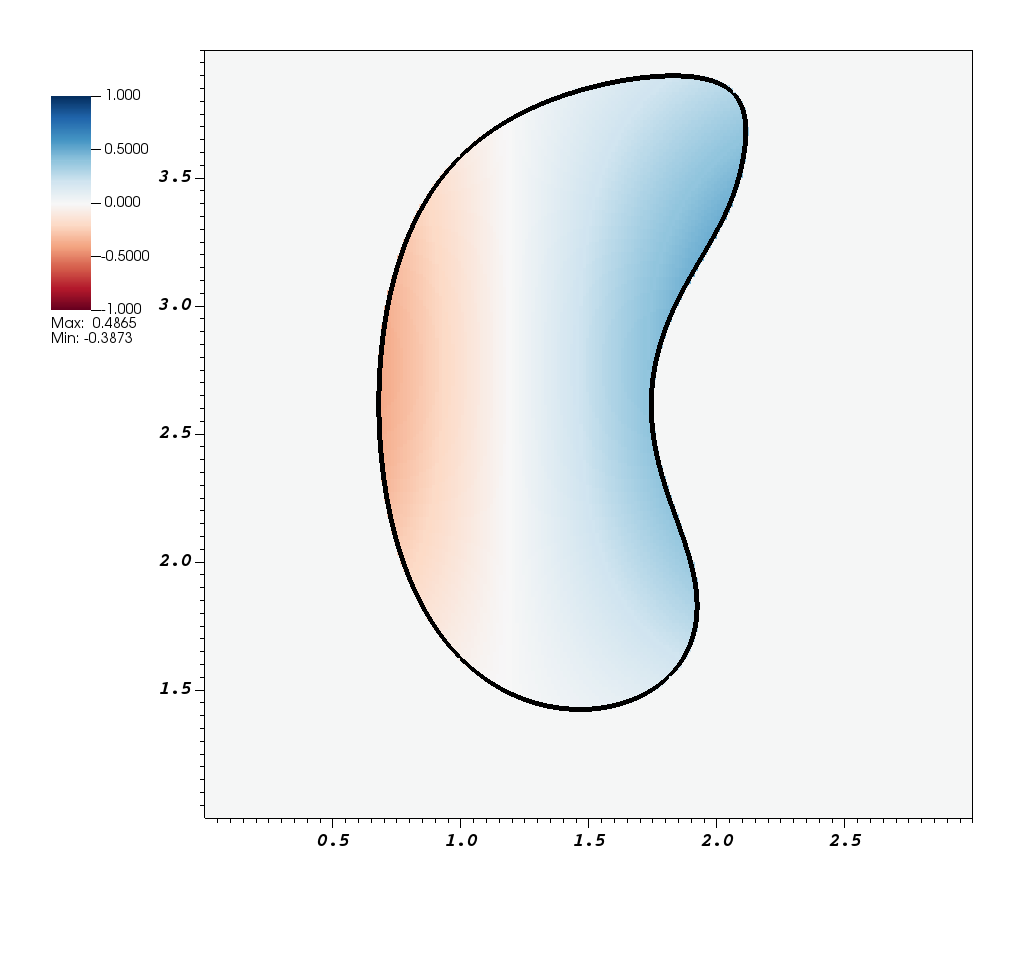}
		\caption{Test-2: $t=1.2$}
	\end{subfigure}
	\begin{subfigure}{0.3\linewidth}
		\centering
		\includegraphics[width=0.9\linewidth]{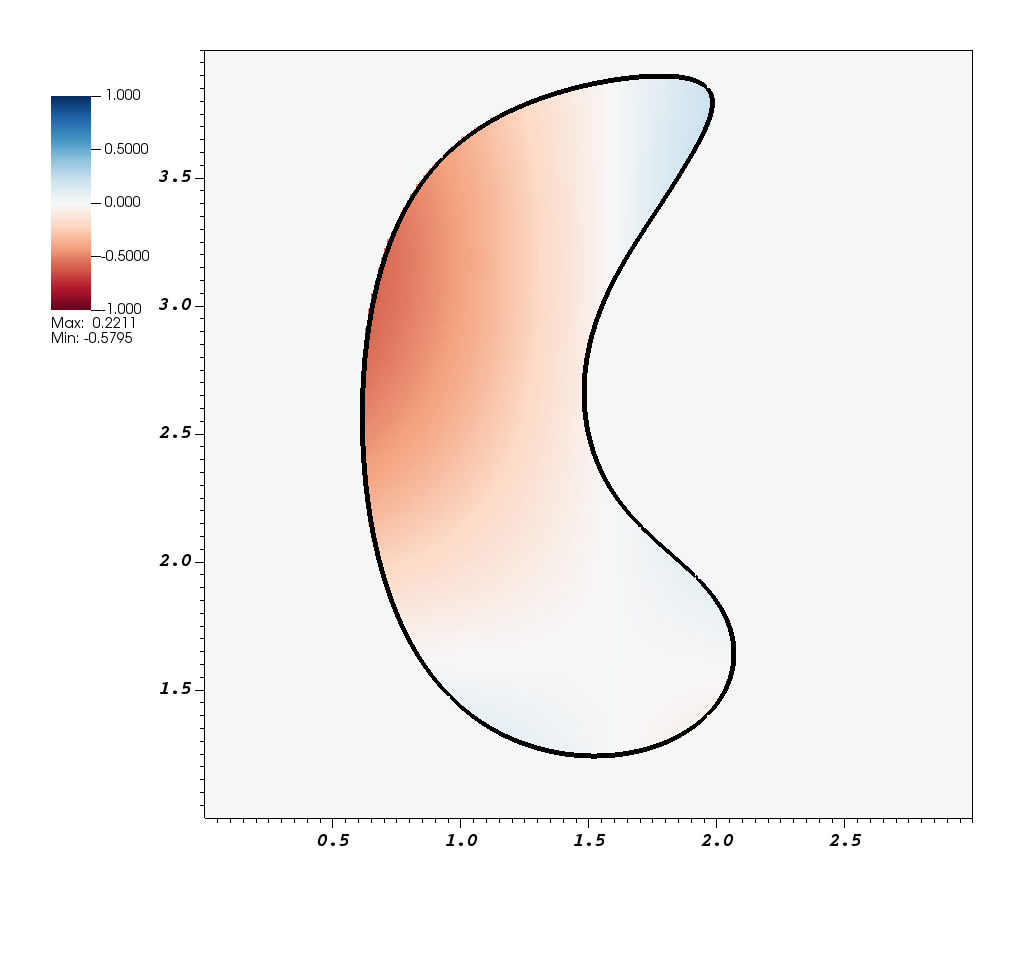}
		\caption{Test-2: $t=1.6$}
	\end{subfigure}
	\begin{subfigure}{0.3\linewidth}
		\centering
		\includegraphics[width=0.9\linewidth]{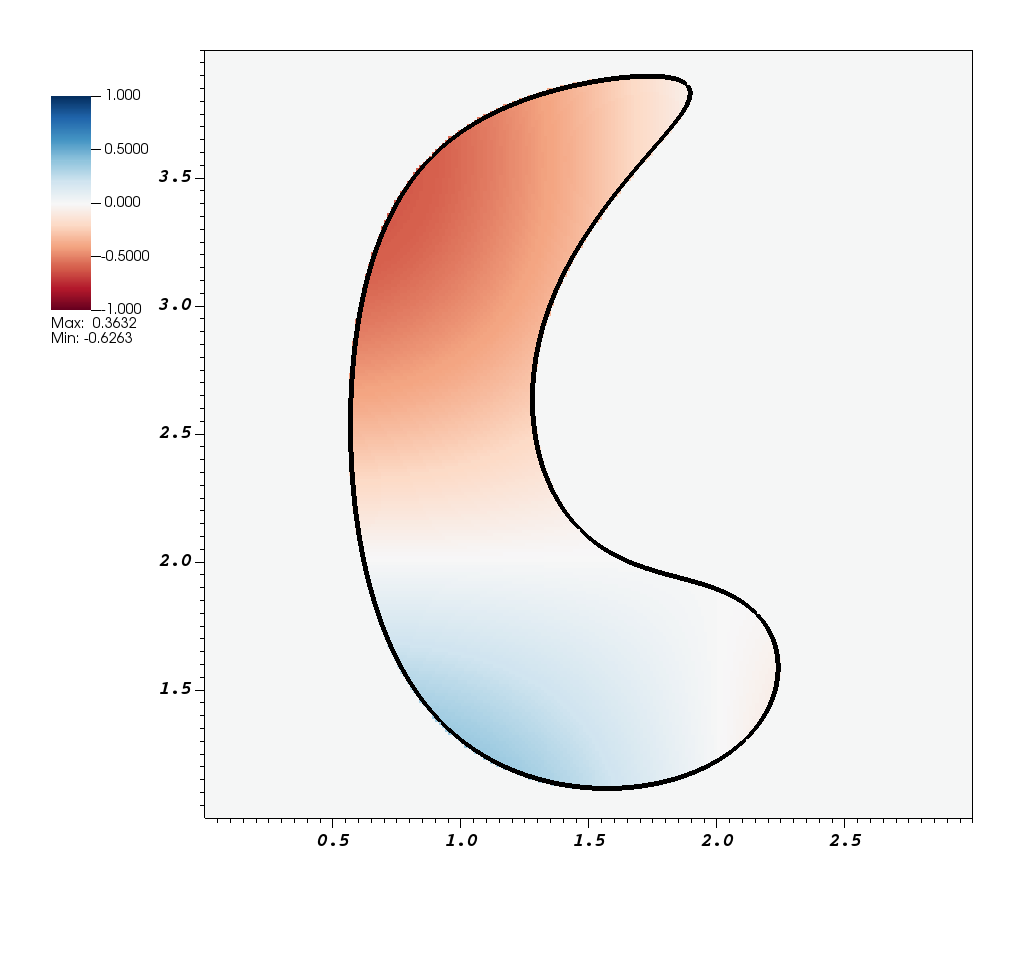}
		\caption{Test-2: $t=2$}
	\end{subfigure}
    \caption{Numerical concentration fields at different times. 
    The moving boundary is depicted in a black curve. 
    The simulation is run to a final time of $T=2$ 
     with a grid resolution of $h=1/80$ 
     and the CFL number of $C_\mathrm{CFL}\approx0.05$.}
    \label{fig:VortexFlow}
\end{figure}

\section{Conclusions}
\label{sec:Conclusions}
We have proposed a fourth-order accurate algorithm 
 for solving the two-dimensional advection-diffusion equations 
 with moving boundaries. 
The introduction of Yin sets and the ARMS technique 
 enables us to explicitly and accurately model complex moving regions, 
 while the incorporation of PLG algorithm 
 enables the handling of finite volume spatial discretization 
 on irregular computational regions. 

A key innovation of this work is the cell-merging technique, 
 which enables high-order accurate cut-cell methods for 
 partial differential equations(PDEs) with moving boundaries. 
This technique not only maintains temporal continuity for PDEs 
 using cell-integral values as the time evolution variable, 
 but it also applicable to PDEs 
 using face-integral values as the time evolution variable, 
 such as solving the incompressible Navier-Stokes equations(INSE) 
 with moving boundaries through cut-cell method, 
 which is exactly our immediate research focus. 
Subsequently, we plan to integrate these methods 
 to address the Boussinesq equations, 
 which will enable us to investigate intriguing phenomena 
 such as the St.Andrew's cross.


\begin{thebibliography}{10}

\bibitem{Ahn2007MOF}
{\sc H.~T. Ahn and M.~Shashkov}, {\em Multi-material interface reconstruction on generalized polyhedral meshes}, J. Comput. Phys., 226 (2007), pp.~2096--2132.

\bibitem{Barrett2022}
{\sc A.~Barrett, A.~L. Fogelson, and B.~E. Griffith}, {\em A hybrid semi-lagrangian cut cell method for advection-diffusion problems with {R}obin boundary conditions in moving domains}, Journal of Computational Physics, 449 (2022), p.~110805.

\bibitem{Boiarkine2011ConvDiffALE}
{\sc O.~Boiarkine, D.~Kuzmin, S.~Čanić, G.~Guidoboni, and A.~Mikelić}, {\em A positivity-preserving {ALE} finite element scheme for convection–diffusion equations in moving domains}, Journal of Computational Physics, 230 (2011), pp.~2896--2914.

\bibitem{Causon2001Cell_Merging}
{\sc D.~Causon, D.~Ingram, and C.~Mingham}, {\em A cartesian cut cell method for shallow water flows with moving boundaries}, Advances in Water Resources, 24 (2001), pp.~899--911.

\bibitem{Cremonesi2020PFEM_Review}
{\sc M.~Cremonesi, A.~Franci, S.~R. Idelsohn, and E.~O{\~n}ate}, {\em A state of the art review of the particle finite element method (pfem)}, Archives of Computational Methods in Engineering, 27 (2020), pp.~1709 -- 1735.

\bibitem{Fei2018}
{\sc L.~Fei, K.~H. Luo, C.~Lin, and Q.~Li}, {\em Modeling incompressible thermal flows using a central-moments-based lattice {B}oltzmann method}, International Journal of Heat and Mass Transfer, 120 (2018), pp.~624--634.

\bibitem{Formaggia2004ALE}
{\sc L.~Formaggia and F.~Nobile}, {\em Stability analysis of second-order time accurate schemes for {ALE–FEM}}, Computer Methods in Applied Mechanics and Engineering, 193 (2004), pp.~4097--4116.

\bibitem{Glowinski2001}
{\sc R.~Glowinski, T.~Pan, T.~I. Helsa, D.~Joseph, and J.~P{\'e}riaux}, {\em A fictitious domain approach to the direct numerical simulation of incompressible viscous flow past moving rigid bodies: application to particulate flow}, Journal of Computational Physics, 169 (2001), pp.~363--426.

\bibitem{Helzel2005H_Box}
{\sc C.~Helzel, M.~J. Berger, and R.~J. Leveque}, {\em A high-resolution rotated grid method for conservation laws with embedded geometries}, SIAM Journal on Scientific Computing, 26 (2005), pp.~785--809.

\bibitem{Hirt1981VOF}
{\sc C.~W. Hirt and B.~D. Nichols}, {\em Volume of fluid ({VOF}) method for the dynamics of free boundaries}, J. Comput. Phys., 39 (1981), pp.~201--225.

\bibitem{Hu2024ARMS}
{\sc D.~Hu, K.~Liang, L.~Ying, S.~Li, and Q.~Zhang}, {\em {ARMS}: Adding and removing markers on splines for high-order general interface tracking under the {MARS} framework}, Journal of Computational Physics, 521 (2024), p.~113574.

\bibitem{Idelsohn2006PFEM}
{\sc S.~Idelsohn, E.~Oñate, F.~D. Pin, and N.~Calvo}, {\em Fluid–structure interaction using the particle finite element method}, Computer Methods in Applied Mechanics and Engineering, 195 (2006), pp.~2100--2123.

\bibitem{Kennedy2003IMEX}
{\sc C.~A. Kennedy and M.~H. Carpenter}, {\em Additive {R}unge–{K}utta schemes for convection–diffusion–reaction equations}, Applied Numerical Mathematics, 44 (2003), pp.~139--181.

\bibitem{Kim2006IBM}
{\sc D.~Kim and H.~Choi}, {\em Immersed boundary method for flow around an arbitrarily moving body}, Journal of Computational Physics, 212 (2006), pp.~662--680.

\bibitem{Leveque1994IIM}
{\sc R.~J. LeVeque and Z.~Li}, {\em The immersed interface method for elliptic equations with discontinuous coefficients and singular sources}, SIAM Journal on Numerical Analysis, 31 (1994), pp.~1019--1044.

\bibitem{LeVeque1997IIM}
{\sc R.~J. LeVeque and Z.~Li}, {\em Immersed interface methods for stokes flow with elastic boundaries or surface tension}, SIAM Journal on Scientific Computing, 18 (1997), pp.~709--735.

\bibitem{LiZhilin1997IMM_MvInterface}
{\sc Z.~Li}, {\em Immersed interface methods for moving interface problems}, Numerical Algorithms, 14 (1997), pp.~269--293.

\bibitem{Liao2010IBM_MvBdry}
{\sc C.-C. Liao, Y.-W. Chang, C.-A. Lin, and J.~M. McDonough}, {\em Simulating flows with moving rigid boundary using immersed-boundary method}, Computers \& Fluids, 39 (2010), pp.~152--167.

\bibitem{Meinke2013}
{\sc M.~Meinke, L.~Schneiders, C.~Günther, and W.~Schröder}, {\em A cut-cell method for sharp moving boundaries in cartesian grids}, Computers \& Fluids, 85 (2013), pp.~135--142.

\bibitem{Peskin1972IBM}
{\sc C.~S. Peskin}, {\em Flow patterns around heart valves: A numerical method}, Journal of Computational Physics, 10 (1972), pp.~252--271.

\bibitem{Peskin2002IBM}
{\sc C.~S. Peskin}, {\em The immersed boundary method}, Acta Numerica, 11 (2002), pp.~479 -- 517.

\bibitem{Russel2003IIM}
{\sc D.~Russell and Z.~{Jane Wang}}, {\em A cartesian grid method for modeling multiple moving objects in {2D} incompressible viscous flow}, Journal of Computational Physics, 191 (2003), pp.~177--205.

\bibitem{Schneiders2013}
{\sc L.~Schneiders, D.~Hartmann, M.~Meinke, and W.~Schröder}, {\em An accurate moving boundary formulation in cut-cell methods}, Journal of Computational Physics, 235 (2013), pp.~786--809.

\bibitem{Sharan2003}
{\sc M.~Sharan and S.~G. Gopalakrishnan}, {\em Mathematical Modeling of Diffusion and Transport of Pollutants in the Atmospheric Boundary Layer}, Birkh{\"a}user Basel, 2003, pp.~357--394.

\bibitem{Tryggvason2001FT}
{\sc G.~Tryggvason, B.~Bunner, D.~Juric, W.~Tauber, S.~Nas, J.~Han, N.~Al-Rawahi., and Y.-J. Jan}, {\em A front-tracking method for the computations of multiphase flow}, J. Comput. Phys., 169 (2001), pp.~708--759.

\bibitem{Xu2006IIM_MvBdry}
{\sc S.~Xu and Z.~J. Wang}, {\em An immersed interface method for simulating the interaction of a fluid with moving boundaries}, Journal of Computational Physics, 216 (2006), pp.~454--493.

\bibitem{Yu2006}
{\sc Z.~Yu, X.~Shao, and A.~Wachs}, {\em A fictitious domain method for particulate flows with heat transfer}, Journal of Computational Physics, 217 (2006), pp.~424--452.

\bibitem{ZhangLT2007IFEM}
{\sc L.~Zhang and M.~Gay}, {\em Immersed finite element method for fluid-structure interactions}, Journal of Fluids and Structures, 23 (2007), pp.~839--857.

\bibitem{ZhangLucy2004IFEM}
{\sc L.~Zhang, A.~Gerstenberger, X.~Wang, and W.~K. Liu}, {\em Immersed finite element method}, Computer Methods in Applied Mechanics and Engineering, 193 (2004), pp.~2051--2067.

\bibitem{Zhang2018CubicMARS}
{\sc Q.~Zhang}, {\em Fourth- and higher-order interface tracking via mapping and adjusting regular semianalytic sets represented by cubic splines}, SIAM Journal on Scientific Computing, 40 (2018), pp.~A3755--A3788.

\bibitem{Zhang2014IPAM}
{\sc Q.~Zhang and A.~Fogelson}, {\em Fourth-order interface tracking in two dimensions via an improved polygonal area mapping method}, SIAM Journal on Scientific Computing, 36 (2014), pp.~A2369--A2400.

\bibitem{Zhang2016MARS}
{\sc Q.~Zhang and A.~Fogelson}, {\em {MARS}: An analytic framework of interface tracking via mapping and adjusting regular semialgebraic sets}, SIAM Journal on Numerical Analysis, 54 (2016), pp.~530--560.

\bibitem{Zhang2012AdvDiff}
{\sc Q.~Zhang, H.~Johansen, and P.~Colella}, {\em A fourth-order accurate finite-volume method with structured adaptive mesh refinement for solving the advection-diffusion equation}, SIAM Journal on Scientific Computing, 34 (2012), pp.~B179--B201.

\bibitem{Zhang2020YinSet}
{\sc Q.~Zhang and Z.~Li}, {\em Boolean algebra of two-dimensional continua with arbitrarily complex topology}, Math. Comput., 89 (2020), pp.~2333--2364.

\bibitem{Zhang2024PLG}
{\sc Q.~Zhang, Y.~Zhu, and Z.~Li}, {\em An {AI}-aided algorithm for multivariate polynomial reconstruction on {C}artesian grids and the {PLG} finite difference method}, Journal of Scientific Computing, 101 (2024), p.~66.

\bibitem{Zhang2024Elliptic}
{\sc Y.~Zhu, Z.~Li, and Q.~Zhang}, {\em A fast fourth-order cut cell method for solving elliptic equations in two-dimensional irregular domains}, In Progreess.

\end{thebibliography}

\bibliographystyle{siamplain}

\end{sloppypar}
\end{document}